\documentclass[11pt]{amsart}
\usepackage{amsmath,amssymb,latexsym,esint,cite,mathrsfs}
\usepackage{verbatim,wasysym}
\usepackage[left=2.3cm,right=2.3cm,top=2.75cm,bottom=2.75cm]{geometry}
\usepackage{tikz,enumitem,graphicx, subfig, microtype, color}
\usepackage{epic,eepic}

\usepackage[colorlinks=true,urlcolor=blue, citecolor=red,linkcolor=blue,
linktocpage,pdfpagelabels, bookmarksnumbered,bookmarksopen]{hyperref}
\usepackage[hyperpageref]{backref}
\usepackage[english]{babel}
\allowdisplaybreaks
\numberwithin{equation}{section}
\newtheorem{theorem}{Theorem}[section]
\newtheorem{proposition}[theorem]{Proposition}

\newtheorem{lemma}[theorem]{Lemma}
\theoremstyle{definition}

\newtheorem{remark}[theorem]{Remark}

\newcommand{\la}{\lambda}
\newcommand{\R}{\mathbb{R}}

\newcommand{\supp}{{\rm supp}{\hspace{.05cm}}}

\newcommand{\intR}{\displaystyle\int\limits_{\mathbb{R}^2}}
\begin{document}
\title
 [Supercritical Schr\"{o}dinger-Poisson system with steep potential well]
 {Planar Schr\"{o}dinger-Poisson system with steep \\  potential well: supercritical exponential case}

\author[L.\ Shen]{Liejun Shen}

\author[M.\ Squassina]{Marco Squassina}

 \address{Liejun Shen, \newline\indent Department of Mathematics, Zhejiang Normal University, \newline\indent
	Jinhua, Zhejiang, 321004, People's Republic of China}
\email{\href{mailto:ljshen@zjnu.edu.cn}{ljshen@zjnu.edu.cn}}

\address{Marco Squassina, \newline\indent
	Dipartimento di Matematica e Fisica \newline\indent
	Universit\`a Cattolica del Sacro Cuore, \newline\indent
	Via della Garzetta 48, 25133, Brescia, Italy}
\email{\href{mailto:marco.squassina@unicatt.it}{marco.squassina@unicatt.it}}

\thanks{L.\ Shen was partially supported by NSFC (12201565). M.\  Squassina  is  member  of  Gruppo  Nazionale  per
	l'Analisi  Matematica,  la Probabilita  e  le  loro  Applicazioni  (GNAMPA)  of  the  Istituto  Nazionale  di  Alta  Matematica  (INdAM)}

\begin{abstract}
We study a class of planar Schr\"{o}dinger-Poisson systems
\[
 \left\{%
\begin{array}{ll}
     -\Delta u+\lambda V(x)u+\phi u=f(u) , & x\in\R^2, \\
     \Delta \phi=u^2, &  x\in\R^2,\\
\end{array}%
\right.
\]
where $\lambda>0$ is a parameter, $V\in \mathcal{C}(\R^2,\R^+)$ has a potential well $\Omega \triangleq\text{int}\, V^{-1}(0)$
and the nonlinearity $f$ fulfills the
 supercritical exponential growth at infinity in the Trudinger-Moser sense.
 By exploiting the mountain-pass theorem and elliptic regular theory, we establish the existence and concentrating behavior
 of ground state solutions for sufficiently large $\lambda$.
\end{abstract}
\maketitle

%\begin{center}
%	\begin{minipage}{8.5cm}
%		\small
%		\tableofcontents
%	\end{minipage}
%\end{center}
%
%\smallskip

\section{Introduction and main results}\label{Introduction}

In this paper, we focus on the existence and concentrating behaviour of ground state solutions for the
following planar Schr\"{o}dinger-Poisson system
 \begin{equation}\label{mainequation1}
\left\{%
\begin{array}{ll}
     -\Delta u+\lambda V(x)u+\phi u=f(u) , & x\in\R^2, \\
     \Delta \phi=u^2, &  x\in\R^2,\\
\end{array}%
\right.
\end{equation}
where $\lambda>0$ is a parameter, $V\in \mathcal{C}(\R^2,\R^+)$ has a potential well $\Omega \triangleq\text{int}\,V^{-1}(0)$
and the nonlinearity $f$ fulfills the
 \emph{supercritical exponential growth} at infinity in the Trudinger-Moser sense.
More precisely, throughout the whole paper,
 we shall suppose
 the nonlinearity $f$ to have the form
\begin{equation}\label{form}
f(t)=h(t) e^{\alpha |t|^\tau},~\quad\forall t\in\R,
\end{equation}
for some $\alpha>0$ and $\tau\geq2$. In what follows, we assume that $h:\R \to \R$ is a function satisfying:
\begin{itemize}
  \item[$(h_1)$] $h\in \mathcal{C}(\R)$ vanishes for every $t\in(-\infty,0]$ and $h(t)=o(t)$ as $t\to0$;
  \item[$(h_2)$] The function $h(t)/ t^3$ is increasing for all $t>0$;
\item[$(h_3)$] There exist $\delta \in (0,2)$ and $\gamma,M>0$ such that $|h(t)| \leq M(e^{\gamma |t|^{\delta}}-1)$
for each $t\in\R^+=(0,+\infty)$.
\end{itemize}

Over the past several decades, a lot of attentions have been paid to the standing, or solitary, wave solutions of Schr\"{o}dinger-Poisson
systems of the type
\begin{equation}\label{standing}
\left\{%
\begin{array}{ll}
     i\frac{\partial\psi}{\partial t}=-\Delta \psi+W(x)\psi+m\phi \psi-
     \widetilde{f}(|\psi|)\psi, & \text{in}~
     \R^+\times\R^d, \\
     \Delta \phi=|\psi|^2, & \text{in}~  \R^d,\\
\end{array}%
\right.
\end{equation}
where $\psi:\R^d\times\R\to \mathbb{C}$ denotes the time-dependent wave function, $W:\R^d\to\R$ is a real external potential,
$m\in\R$ is a parameter, $\phi$ represents an internal
potential for a nonlocal self-interaction of the wave function
 and the nonlinear term $f(\psi)\triangleq \widetilde{f}(|\psi|)\psi$ describes the interaction effect among particles.
 Let $\psi(x,t)=\exp(-i\omega t)u(x)$
with $\omega\in\R$, then the standing wave solutions of \eqref{standing} also lead to the Schr\"{o}dinger-Poisson system
\begin{equation}\label{introduction1}
 \left\{%
\begin{array}{ll}
     -\Delta u+V(x)u+m\phi u=
     f(u) , & \text{in}~
       \R^d, \\
     \Delta \phi=u^2, &  \text{in}~  \R^d,\\
\end{array}%
\right.
\end{equation}
where and in the sequel $\bar{V}(x)=W(x)+\omega$.
In view of the paper \cite{Cingolani}, the second equation in \eqref{introduction1} determines $\phi: \R^d\to\R$ only up to harmonic functions.
Conversely, it is natural to choose $\phi$ as the negative Newton potential of $u^2$, that is, the convolution of $u^2$ with the
fundamental solution $\Phi_d$ of the Laplacian, which is denoted by
$\Phi_d(x)=-1/(d(d-2)\omega_d)|x|^{2-d}$ if $d\geq3$, and $\Phi_2(x)=-\frac{1}{2\pi}\log(|x|)$ if $d=2$, here
$\omega_d$ denotes the volume of the unit ball in $\R^d$. With this inversion of the second equation in \eqref{introduction1}, we 
arrive at the integro-differential equation
\begin{equation}\label{introduction2}
-\Delta u+\bar{V}(x)u+m \big(\Phi_d\ast u^2\big)u=  f(u),~ \text{in}~ \R^d.
\end{equation}

When $m\neq0$, the Poisson term $(\Phi_d\ast u^2)u$
causes that \eqref{introduction2} is not a pointwise identity any longer such that there are some mathematical difficulties
 which make the study of it more interesting. In the three dimensional case,
 because of the relevances in physics,
we have a rich literature associated with \eqref{introduction2}
 and its generalizations under
the variant assumptions on $\bar{V}$ and $f$ via variational methods,
the interested reader can
  refer to \cite{Ackermann,Ambrosetti,CM,DMT,MZ,Ruiz} and the references therein.

For $d=2$, the Schr\"{o}dinger-Poisson equation \eqref{introduction2} can be rewritten as the form
\begin{equation}\label{introduction3}
-\Delta u+\bar{V}(x)u+\frac{m}{2\pi} \big(\log(|x|)\ast u^2\big)u=f(u)~\quad \text{in}~  \R^2,
\end{equation}
whose variational functional is defined by
\[
I  (u) =\frac12\int_{\R^2}\big[|\nabla u|^2+\bar{V}(x)u^2\big]dx+\frac{m}{8\pi} \int_{\R^2}\int_{\R^2}\log(|x-y|)u^2(x)u^2(y)dxdy
  -\int_{\R^2}\int_{0}^uf(t)dtdx.
\]
As mentioned by Stubbe in \cite{Stubbe}, the functional $I$ would not be well-defined on $H^1(\R^2)$.
To overcome this difficulty, the author introduced a new Hilbert space
\[
X=\bigg\{u\in H^1(\R^2):
\int_{\R^2}\log(1+|x|)u^2 dx<+\infty\bigg\},
\]
endowed with the inner product and norm
\[
(u,v)_X=\int_{\R^2}\big[\nabla u\nabla v+uv+\log(1+|x|)uv\big]dx~\text{and}~
\|u\|_X=\sqrt{(u,u)_X}.
\]
In Stubbe's argument, it depends strongly on the crucial identity
\[
\log r=\log(1+r)-\log\bigg(1+\frac1r\bigg),~\forall r>0,
\]
because it permits us to define the variational functionals $V_1,V_2: X\to\R$ by
\[
V_1(u)\triangleq\int_{\R^2}\int_{\R^2}\log(1+|x-y|)u^2(x)u^2(y)dxdy,~ \forall u\in X,
\]
and
\[
V_2(u)\triangleq\int_{\R^2}\int_{\R^2}\log\bigg(1+\frac{1}{|x-y|}\bigg)u^2(x)u^2(y)dxdy,~  \forall u\in X.
\]
Stubbe \cite{Stubbe} proved that $V_1,V_2\in \mathcal{C}^1(X,\R)$ and found the equality
\[
V_0(u)\triangleq\int_{\R^2}\int_{\R^2}\log(|x-y|)u^2(x)u^2(y)dxdy=V_1(u)-V_2(u),~ \forall u\in X,
\]
which implies that $I$ given in \eqref{introduction3} is of class $\mathcal{C}^1(X)$.

In \cite{Cingolani}, Cingolani and Weth made full use of the above variational framework to study the existence and multiplicity
of nontrivial solutions for the equation
\begin{equation}\label{CW}
 -\Delta u+ u+\frac{m}{2\pi} \big(\log(|x|)\ast u^2\big)u=b|u|^{p-2}u~\quad \text{in}~  \R^2,
\end{equation}
where $b\geq0$ and $p\geq4$. With the help of variational methods and action of group,
they investigated
  the existence of multiple solutions
 and obtained a ground state solution by minimizing the energy functional over the Nehari
manifold for Eq. \eqref{CW}. Subsequently, considering \eqref{CW} with $b=1$ and $p\in(2,4]$,
 Du-Weth \cite{Du} proved that it possesses a nontrivial solution if $p\in(2,3)$
 and ground state solution if $p\in(3,4]$ by constructing a asymptotic (PS) sequence,
 where the Poho\u{z}aev
identity was established. In \cite{BCV}, Bonheure \emph{et al.} concluded the asymptotic decay of the unique positive, radially
symmetric solution to \eqref{CW} with $b=0$. Afterwards, Chen, Shi and Tang \cite{CST} extended the main results in \cite{Du} to a
general nonlinearity.
It is important to note that the above cited papers relied on the fact that the potential is constant or
$\mathbb{Z}^2$-periodic. To circumvent the obstacle
the authors in \cite{CT1,CT2}, restricting
in an axially symmetric space which is weaker than the classic
radially symmetric one,
succeeded in finding nontrivial solutions for
Eq. \eqref{CW} with a non-constant or non-periodic potential.
Concerning some other meaningful research works associated with
 Eq. \eqref{CW},
we suggest the reader to look at \cite{Azzollini,CJ,DFJ} and their references
even if these references are far to be exhaustive.

Another interesting point for considering Schr\"{o}dinger-Poisson equations, like Eq. \eqref{introduction3},
is the space dimension $d=2$. Moreover, to our best
knowledge, the existence result for the Schr\"{o}dinger-Poisson equation with steep potential well
has not been studied yet, especially with the nonlinearity involving the \emph{supercritical exponential growth}.
Therefore, one of the main purposes of the present paper is to fill these blanks. Generally speaking,
we aim at establishing the existence and concentrating behavior of ground state solutions
for Eq. \eqref{mainequation1}.

As known, for every bounded domain $\Omega\subset\R^2$, the imbedding $H_0^1(\Omega)\hookrightarrow L^p(\Omega)$ with $1\leq p<+\infty$
does not imply that
$H_0^1(\Omega)\hookrightarrow L^\infty(\Omega)$.
As a consequence, one naturally wonders if there
exists another kind of \emph{maximal growth} in this situation.
Indeed,
the authors in \cite{PSI,TNS,MJ}
 got the following sharp maximal exponential integrability for
functions in $H_{0}^{1}(\Omega)$:
\begin{equation}\label{TM}
\sup\limits_{u\in H_{0}^{1}(\Omega),\, \|\nabla u\|_{L^2
(\Omega)}\leq1}\int_{\Omega}e^{\alpha u^{2}}dx\leq C\text{meas}(\Omega)~\,\,\,\,\mbox{if}~\alpha\leq4\pi,
\end{equation}
where the constant $C=C(\alpha)>0$ and $\text{meas}(\Omega)$ stands for
the Lebesgue measure of $\Omega$. As we known,
 \eqref{TM} is the celebrated Trudinger-Moser type inequality
related with elliptic problems with \emph{critical exponential growth}
 saying that a function $f$ satisfies the following: there exists
$\alpha_{0}>0$ such that
\begin{equation}\label{definition}
\lim\limits_{|t|\rightarrow+\infty}
\frac{|f(t)|}{e^{\alpha t^{2}}}=
\left\{
  \begin{array}{ll}
    0, & \forall
\alpha>\alpha_{0}, \\
    +\infty, &\forall \alpha<\alpha_{0}.
  \end{array}
\right.
\end{equation}
This definition was introduced by Adimurthi and Yadava \cite{AYA},
see also de Figueiredo, Miyagaki and Ruf  \cite{Figueiredo} for example.
Motivated by the previous works in \cite{AS1,AS2}, there exist two ways to understand that
the function $h$, defined in \eqref{form} together with $(h_3)$,
 satisfies the so-called \emph{{supercritical exponential growth}} in the following sense:
\begin{equation}\label{definition2}
(\text{I})~\tau > 2~\text{is arbitrary and}~\alpha>0~\text{is fixed;}~\,\,\,
(\text{II})~\alpha>0~\text{is arbitrary and}~\tau\geq2~\text{is fixed.}
\end{equation}
Moreover, one could call Cases (I) and (II) in \eqref{definition2} to be the
\emph{{subcritical-supercritical exponential growth}}
and \emph{{critical-supercritical exponential growth}}, respectively.

As to the whole space $\R^2$, the author in \cite{Bezerra1}
established the following version of the Trudinger-Moser inequality
(see also \cite{Cao} for example):
\begin{equation}\label{TM1}
e^{\alpha u^2}-1\in L^2(\R^2),~\forall \alpha>0~\text{and}~u\in H^1(\R^2).
\end{equation}
Moreover, for all $u\in H^1(\R^2)$ with $\|u\|_{L^2(\R^2)}\leq M<+\infty$, there is a $C=C(M,\alpha)>0$
such that
\begin{equation}\label{TM2}
\sup\limits_{u\in H^{1}(\R^2),\,\,\|\nabla u\|_{L^2
(\R^2)}\leq1}\int_{\R^2}\big(e^{\alpha u^{2}}-1\big)dx\leq C ~\mbox{if}~\alpha<4\pi.
\end{equation}
Concerning some other generalizations, extensions and applications
of the Trudinger-Moser inequalities for bounded and unbounded domains,
we refer to  \cite{Figueiredo,LR,DSD,Ruf} and the references therein.

As far as we are concerned,
 Alves and Figueiredo \cite{AF} firstly applied \eqref{definition}
 to the Schr\"odinger-Poisson equation \eqref{introduction3}
 and investigated the existence of ground state solutions by Nehari manifold method.
Along this direction, there are more and more research works
concerning this topic including the two dimensional Choquard problem with logarithmic kernel,
see \cite{CSTW,CT,CT2,Shen,ACFM,LRTZ,LRZ,LRZ2,CRT} for example.

Before stating the main results in this paper, on $V$
we shall impose the following assumptions.
\begin{itemize}
  \item[$(V_1)$] \emph{$V\in \mathcal{C}(\R^2,\R)$ with $V(x)\geq 0$ on $\R^2$;}
  \item[$(V_2)$] \emph{$\Omega\triangleq\text{\emph{int}}V^{-1}(0)$
is nonempty and bounded with smooth boundary, and $\overline{\Omega}=V^{-1}(0)$;}
  \item[$(V_3)$] \emph{there exists $b> 0$ such that $\Xi\triangleq\{x\in \R^2
: V(x) <b\}$ is nonempty and has finite measure.}
\end{itemize}

The assumptions like $(V_1)-(V_3)$ were initiated by Bartsch \emph{et al}. in
 \cite{Bartsch2,Bartsch1}. Particularly, the harmonic trapping potential
\[
V(x)=\left\{
       \begin{array}{ll}
         \omega_1|x_1|^2+\omega_2|x_2|^2-\omega, &\text{if}~
|(\sqrt{\omega_1}x_1,\sqrt{\omega_2}x_2)|^2\geq \omega, \\
         0, &\text{if}~|(\sqrt{\omega_1}x_1,\sqrt{\omega_2}x_2)|^2\leq \omega ,
       \end{array}
     \right.
\]
with $\omega> 0$ satisfies $(V_1)-(V_3)$, where $\omega_i>0$
is called by the anisotropy factor of the trap in quantum physics and trapping frequency of the
$i$th-direction in mathematics, see e.g. \cite{Lushnikov2,Carles,Bellazzini}.
Indeed, the potential $\lambda V(x)$ with the above
hypotheses is usually known as the steep potential well.

For all fixed $\lambda>0$, by $(V_1)$, define the Hilbert space
\[
E_\lambda \triangleq\bigg\{u\in L^{2}_{\text{loc}}(\R^2):|\nabla u|\in L^2(\R^2)~\text{and}\int_{\R^2} \lambda V(x)|u|^2dx<+\infty\bigg\}
\]
equipped with the inner product and norm
\[
(u,v)_{E_\lambda}=\int_{\R^2}\big[\nabla u\nabla v+\lambda V(x)uv \big]dx~\text{and}~
\|u\|_{E_\lambda}=\sqrt{(u,u)_{E_\lambda}}, ~\forall u,v\in {E_\lambda}.
\]
Hereafter,
let us denote $E$ and $\|\cdot\|_E$ by $E_\lambda$ and $\|\cdot\|_{E_\lambda}$ for $\lambda=1$,
respectively. It is simple to observe that $\|\cdot\|_E\leq \|\cdot\|_{E_\lambda}$
for every $\lambda\geq1$. Therefore, due to Lemma \ref{imbedding} below,
$E_\lambda$ could be continuously imbedded into $H^1(\R^2)$ and then into $X$.
With these discussions, we could introduce the work space
\[
X_\lambda \triangleq\bigg\{u\in X:\int_{\R^2} \lambda V(x)|u|^2dx<+\infty\bigg\}
\]
which is a Hilbert space equipped with
the inner product and norm
\[
(u,v)_{X_\lambda}=\int_{\R^2}\big[\nabla u\cdot \nabla v+(\lambda V(x)+\log(1+|x|))uv \big]dx~\text{and}~
\|u\|_{X_\lambda} =\sqrt{(u,u)_{X_\lambda}}, ~\forall u,v\in {X_\lambda}.
\]
Obviously, $\|\cdot\|_{X_\lambda}=\sqrt{\|\cdot\|^2_{E_\lambda}+\|\cdot\|_*^2}$,
where $\|u\|_*= (\int_{\R^2} \log(1+|x|)|u|^2dx )^{\frac12}$ for all $u\in X$.

\smallskip
Now, we shall exhibit the first main result in this paper as follows.

\begin{theorem}\label{maintheorem1}
Let $V$ satisfy $(V_1)-(V_3)$.
Suppose that the nonlinearity $f$ defined in \eqref{form}
requires $(h_1)-(h_3)$, then for each
$\tau>2$, there are $\alpha^*=\alpha^*(\tau)>0$ and $\lambda_0>0$ such that Eq. \eqref{mainequation1}
has a nonnegative ground state solution in $X_\lambda$ for all $\alpha \in (0, \alpha^*)$ and $\lambda>\lambda_0$.
Moreover, if we suppose that
\begin{itemize}
\item[$(h_4)$] there are constants $\xi>0$ and $p>4$ such that $H(t)=\int_0^th(s)ds\geq \xi t^{p}$ for all $t\in[0,1]$,
\end{itemize}
then for every
$\alpha>0$, there exist $\tau_*=\tau_*(\alpha)>2$, $\lambda_0^\prime>0$ and ${\xi}_0>0$
such that Eq. \eqref{mainequation1} possesses a nonnegative ground state solution in $X_\lambda$ for every $\tau \in [2,\tau_*)$,
$\lambda>\lambda_0^\prime$ and ${\xi}>{\xi}_0$.
\end{theorem}

From a variational method point of view, in order to look for nontrivial solutions corresponding to Eq. \eqref{mainequation1},
it would be find critical points for the variational functional $J_\lambda:X_\lambda\to\R$ defined by
\[
{J}_\lambda(u)=\frac{1}{2}\int_{\R^2}[|\nabla u|^2+\lambda V(x)|u|^2]dx+\frac14V_0(u)
- \int_{\R^2} F (u)dx,~\forall u\in X_\lambda.
\]
Alternatively, although the property of $X_\lambda$ is good enough, adopting \eqref{TM1},
one could never show that $\mathcal{J}_\lambda$ is well-defined in $X_\lambda$
which is caused by the appearance the nonlinearity $f$
involving the supercritical exponential growth
\eqref{definition2} in the Trudinger-Moser sense.
As one will see later, this is the biggest difference from \cite{AF,CT2} which prevents
us repeating the approaches in the cited papers to conclude Theorem \ref{maintheorem1}.

Inspired by
\cite{AS1,AS2}, given a fixed constant $R>0$,
  we need to introduce a cutoff function $f^{R,\bar{\delta}}$ and consider an auxiliary equation which involves
  a (sub)critical exponential growth.
Roughly speaking, we define $f^{R,\bar{\delta}}:\R\to\R$ as follows
\begin{equation} \label{fR}
f^{R, \bar{\delta} }(t)=
\left\{
\begin{array}{ll}
	0, &   t \leq 0, \\
	h(t) e^{\alpha t^\tau},&    0\leq t \leq R,\\
	h(t) e^{\alpha R^{\tau-\bar{\delta}}t^{\bar{\delta}}}, &  t \geq R, \\
\end{array}
\right.
\end{equation}
where
\[
 \bar{\delta}\triangleq\left\{
   \begin{array}{ll}
    \delta, & \text{if the Case I in \eqref{definition2} is considered},\\
    2, & \text{if the Case II in \eqref{definition2} is considered}.
    \end{array}
 \right.
\]
In light of such a $f^{R, \bar{\delta}}$, we consider
an auxiliary equation below
\begin{equation}\label{mainequation2}
-\Delta u+\lambda V(x)u+ \big(\log(|x|)\ast u^2\big)u=f^{R,\bar{\delta}}(u)~ \text{in}~  \R^2,
\end{equation}
where and in the sequel $F^{R,\bar{\delta}}(t)=\int_0^tf^{R,\bar{\delta}}(s)ds$.
Obviously, for each fixed $R>0$, by $(h_3)$, it is simple to find that
$f^{R,\bar{\delta}}$ admits a subcritical or critical exponential growth at infinity. Hence,
the variational functional $J_\lambda^{R,\bar{\delta}}:X_\lambda\to\R$ given by
\[
J_\lambda^{R,\bar{\delta}}(u)=\frac{1}{2}\int_{\R^2}[|\nabla u|^2+\lambda V(x)|u|^2]dx+\frac14V_0(u)
- \int_{\R^2} F^{R,\bar{\delta}}(u)dx
\]
 associated with Eq. \eqref{mainequation2} is well-defined and of class $\mathcal{C}^1(X_\lambda)$.
Moreover, we could certify that every critical point of it is a (weak)
solution of Eq. \eqref{mainequation2}.
Recalling \cite[Subsection 4.1]{Willem},
let us call the solution $u_R\in E$ to be a ground state solution of Eq. \eqref{mainequation2}
if it satisfies
\begin{equation}\label{Nehari}
J_\lambda^{R,\bar{\delta}} (u )=\inf_{u\in \mathcal{N}^{R,\bar{\delta}}_\lambda }J_\lambda^{R,\bar{\delta}} (v)\triangleq m^{R,\bar{\delta}}_\lambda,
\end{equation}
where the corresponding Nehari manifold is given as
\[
\mathcal{N}^{R,\bar{\delta}}_\lambda\triangleq\big\{u\in X_\lambda\backslash\{0\}:(J_\lambda^{R,\bar{\delta}})^\prime(u)[u]=0\big\}.
\]
It could easily conclude that if
every ground state solution $u_R\in E$ of Eq. \eqref{mainequation2} satisfying
$|u_R|_\infty\leq R$, then
 $u_R$ is a ground state solution of Eq. \eqref{mainequation1}, where $|\cdot|_q$ stands for the standard $L^q$-norm with $1\leq q\leq \infty$.
 Have this in mind, we should establish such a solution $u_R$
 to derive the proof of Theorem \ref{maintheorem1}. So,
 it is necessary to show the following result.

\begin{theorem}\label{maintheorem2}
Let $V$ satisfy $(V_1)-(V_3)$.
Suppose that the nonlinearity $f$ defined in \eqref{form}
requires $(h_1)-(h_3)$, then for every fixed $R>0$, there exists a $\lambda_0(R)>0$ dependent of $R$ such that Eq. \eqref{mainequation2}
with $\bar{\delta}=\delta$
has a nonnegative ground state solution in $X_\lambda$ for all $\lambda>\lambda_0(R)$.
Moreover, if in addition we suppose that $(h_4)$,
then there exist $\lambda_0^\prime(R)>0$ and ${\xi}_0(R)>0$
such that Eq. \eqref{mainequation2} with $\bar{\delta}=2$ possesses a nonnegative ground state solution in $X_\lambda$ for all
$\lambda>\lambda_0^\prime(R)$ and ${\xi}>{\xi}_0(R)$.
\end{theorem}

\begin{remark}
It is worth pointing out that, up to the best knowledge of us, Theorem \ref{maintheorem2}
seems to be the first results concerning planar Schr\"{o}dinger-Poisson system
with steep potential well involving subcritical or critical exponential growth.
Although there may be existence results for such problem with non-constant or non-periodic
potential, the results in Theorem \ref{maintheorem2} are new in our settings.
We mention the  paper \cite{PSZZ} in which the authors obtained the existence and concentrating
behavior of ground state solutions for the generalized Chern-Simons-Schr\"{o}dinger
equation with steep potential well in the critical exponential case.
Nevertheless, comparing to it,
we have to present some new tricks and analytic skills because of the
indefinite Poisson term $\log(|x|)\ast u^2$ bringing in additionally distinct
challenges.
\end{remark}

Next, we shall contemplate the asymptotical behavior of the ground state solutions
obtained in Theorem \ref{maintheorem1} as $\lambda\to+\infty$. Let $u\in X_\lambda$ be a
ground state solution for Eq. \eqref{mainequation1}, there
is no doubt that it depends on the parameter $\lambda>\lambda_0>0$ (or, $\lambda>\lambda_0^\prime>0$),
so we relabeled it by $u_\lambda$ to emphasize this dependence,
where $\lambda_0>0$ (or, $\lambda_0^\prime>0$) is a constant appearing in the proof of Theorem \ref{maintheorem1} below.
Therefore, we can prove the following result.

\begin{theorem}\label{maintheorem3}
Under the assumptions in Theorem \ref{maintheorem1}, going to a subsequence,
$u_\lambda\to   u_0$ in $X$ as $\lambda\to+\infty$, where $u_0$
is a ground state solution for the Schr\"{o}dinger-Poisson equation
 \begin{equation}\label{mainequation3}
\left\{%
\begin{array}{ll}
\displaystyle     -\Delta u+\bigg(\int_{\Omega}\log|x-y|u^2(y)dy\bigg)   u=f(u) , & x\in\Omega, \\
     u=0, &  x\in\partial\Omega.\\
\end{array}%
\right.
\end{equation}
\end{theorem}

\begin{remark}
Let us recall the main results in Theorems \ref{maintheorem1} and \ref{maintheorem2},
dealing with planar Schr\"{o}dinger-Poisson system
with non-constant or non-periodic potential involving critical exponential growth,
the well-known Ambrosetti-Rabinowitz condition, namely $f(t)t-\mu F(t)\geq0$
with $\mu>4$ for all $t\in\R$, plays a significant role in concluding the boundness of (C)
sequence corresponding to the variational functional.
In view of $(h_2)$ (or, \eqref{ff} below), nevertheless,
 we can remove it successfully by the steep potential well $\lambda V$ under the assumptions $(V_1)-(V_3)$.
\end{remark}

Although we have briefly introduced the essential ideas of the main results in this article,
it is still far from enough to get the detailed proofs of Theorems \ref{maintheorem1}, \ref{maintheorem2} and \ref{maintheorem3}.
Alternatively, the key to considering such supercritical problem
 is to transform it into a subcritical or critical one. As to the $L^\infty$-estimate on nontrivial solutions for a class of
 supercritical problems, the elliptic regular result or Nash-Moser iteration procedure are
  the effective tools. Moreover, we truly anticipate that the results
in the present paper shall be conducive to some more further studies on planar Schr\"{o}dinger-Poisson system
with (sub)critical or supercritical exponential growth.

The outline of the paper is organized as follows. In Section \ref{Preliminaries}, we mainly
present some preliminary results.
In Section \ref{Sec3}, we consider the (sub)critical problem \eqref{mainequation2} and give
the proof of Theorem \ref{maintheorem2}.
Section \ref{Linftyestimate} is devoted to the proofs of Theorems \ref{maintheorem1}
and \ref{maintheorem3}. Some remarks are given in Section \ref{Sec5}.
\\\\
\noindent \textbf{Notations:} From now on in this paper, otherwise mentioned, we ultilize the following notations:
\begin{itemize}
	\item   $C,C_1,C_2,\cdots$ denote any positive constant, whose value is not relevant and $\R^+\triangleq(0,+\infty)$.
	 	\item      Let $(Z,\|\cdot\|_Z)$ be a Banach
space with dual space $(Z^{-1},\|\cdot\|_{Z^{-1}})$, and $\Phi$ be functional on $Z$.
	\item The (C) sequence at a level $c\in\R$ ($(C)_c$ sequence in short)
corresponding to $\Phi$ means that $\Phi(x_n)\to c$ and $(1+\|x_n\|_Z)\|\Phi^{\prime}(x_n)\|_{Z^{-1}}\to 0$ in $X^{-1}$ as $n\to\infty$, where
$\{x_n\}\subset Z$.
	\item  $|\,\,\,|_p$ stands for the usual norm of the Lebesgue space $L^{p}(\mathbb{R}^2)$ for all $p \in [1,+\infty]$, and
	$\Vert\,\,\,\Vert_{H^{i}}$ denotes the usual norm of the Sobolev space $H^{i}(\mathbb{R}^2)$ for $i=1,2$.
\item For any $\varrho>0$ and every $x\in \R^2$, $B_\varrho(x)\triangleq\{y\in \R^2:|y-x|<\varrho\}$.
\item $o_{n}(1)$ denotes the real sequences with $o_{n}(1)\to 0$
 as $n \to +\infty$.
\item $``\to"$ and $``\rightharpoonup"$ stand for the strong and
 weak convergence respectively.
\end{itemize}

\section{Preliminaries}\label{Preliminaries}
In this section, we shall present some preliminary results which shall be exploited frequently
in this paper.
Let us begin defining the following
 three
 auxiliary symmetric bilinear forms $X\times X\to\R$
$$
(u , v) \mapsto B_1(u, v)=\intR \intR \log(1+|x-y|)u(x)v(y) dx dy ,
$$
$$
(u , v) \mapsto B_2(u, v)=\intR \intR \log\left(1+\dfrac{1}{|x-y|}\right)u(x)v(y) dx dy ,
$$
$$
(u , v) \mapsto B(u, v)=B_1(u, v)-B_2(u, v)=\intR \intR\log(|x-y|)u(x)v(y) dx dy.
$$
The above definitions are understood to being over measurable functions $u, v: \R^2 \to \R $, such that the integrals are defined in the
Lebesgue sense. Then, $V_i(u)=B_i(u^2,u^2)$, where
$V_i$ is given as in the Introduction for $i\in\{0,1,2\}$.

\begin{lemma}\label{zz} (\cite[Lemma 2.2]{Cingolani})
\emph{(i)} The space $X$ is compactly embedded in $L^{s}(\R^2)$, for all $s\in [2,\infty)$. \\
\emph{(ii)} $0\leq V_1(u)\leq 2|u|_2^2|u|_*^2\leq 2\|u\|_X^4$ and $V_1$ is weakly semicontinuous in $H^1(\R^2)$.\\
\emph{(iii)} $V_{i}'(u)[v]=4B_i(u^2, uv)$ and so $ V_{i}'(u)[u]=4V_i(u)$, for $u, v\in X$ and
$i\in\{0,1,2\}$. \\
\emph{(iv)} There is $K_0>0$ such that $|B_2(u, v)|\leq K_0 |u|_{\frac{4}{3}}|v|_{\frac{4}{3}}, \forall u, v \in L^{\frac{4}{3}}(\R^2) $. Hence,
	\begin{equation} \label{v2}
	|V_2(u)|\leq K_0|u|_{\frac{8}{3}}^{4} , ~\forall   u\in L^{\frac{8}{3}}(\R^2).
	\end{equation}
\emph{(v) }$V_2$ is completely continuous in $X$, that is,
$$
u_n\rightharpoonup u ~ \mbox{in} ~ X \Longrightarrow V_2(u_n) \to V_2(u).
$$
\end{lemma}

\begin{lemma}\label{yy}
(\cite[Lemma 2.6]{Cingolani}) Let $\{u_n\}$, $\{v_n\}$ and $\{w_n\}$ be bounded sequences in $X$ such that $u_n \rightharpoonup u$ in $X$.
 Then, for every $z\in X$, we have $B_1(v_n w_n,z(u_n - u))\to 0$, as $n\to + \infty$.
\end{lemma}

Taking into account the uniform $L^\infty$-estimate for the nontrivial solutions established in Theorem
\ref{maintheorem2}, we obtain the result below.

\begin{lemma} \label{cc}
It holds that $g_u(x)\triangleq\log(1+ |\cdot|^{-1} )\ast |u|^{2} \in L^{\infty}(\R^2)$ for all $u\in H^1(\R^2)$. Moreover
\begin{equation}\label{13a}
|g_u(x)|_\infty\leq 4\pi(|u|^2_2+|u|_6^2).
\end{equation}
\end{lemma}
\begin{proof}
The original idea comes from \cite{AS0}, we exhibit it in detail for the sake of the
reader's continence. For all $x\in\R^2$, there holds
\begin{align}\label{13b}
\nonumber|g_u(x)| &=\int_{B_1(x)}\log \left(1+\frac{1}{|x-y|}\right)|u (y)|^2 dy+
\int_{\R^2\backslash B _{1}(x)}\log \left(1+\frac{1}{|x-y|}\right)|u (y)|^2 dy \\
   & \leq\int_{B_1(x)} \frac{|u (y)|^2}{|x-y|}  dy+\log2
\int_{\R^2\backslash B _{1}(x)}|u (y)|^2 dy.
\end{align}
It follows from the Holder's inequality that
\begin{equation}\label{13c}
\int_{B_1(x)} \frac{|u_R(y)|^2}{|x-y|}  dy\leq
\bigg(\int_{B_1(x)} \frac{1}{|x-y|^{\frac32}}  dy\bigg)^{\frac23}
\bigg(\int_{\R^2} |u(y)|^6  dy\bigg)^{\frac13}=(4\pi)^{\frac23}\bigg(\int_{\R^2} |u(y)|^6  dy\bigg)^{\frac13}.
\end{equation}
Combining \eqref{13b} and \eqref{13c}, we get the desired result \eqref{13a}. The proof is completed.
\end{proof}

Next, we have to prove the following imbedding result which guarantees
the reasonableness of the work space $X_\lambda$ whence $\lambda\geq1$.

\begin{lemma}\label{imbedding}
Suppose that $(V_1)-(V_3)$, then the imbedding $E\hookrightarrow H^1(\R^2)$ is continuous. In particular,
$E_\lambda$ can be continuously imbedded into $H^1(\R^2)$ for all $\lambda\geq1$.
\end{lemma}

\begin{proof}
For every $u\in H^1(\R^2)$, by adopting $(V_1)$ and $(V_3)$ as well as the Gagliardo-Nirenberg inequality, one has
\begin{align*}
  \int_{\R^2}|u|^2dx & =\int_{\Xi}|u|^2dx+ \int_{\R^2\backslash\Xi}|u|^2dx\leq \sqrt{\text{meas}(\Xi)}
 \bigg(\int_{\R^2}|u|^4dx\bigg)^{\frac12}+\frac1b  \int_{\R^2\backslash\Xi}V(x)|u|^2dx\\
    & \leq \kappa_{\text{GN}}\sqrt{\text{meas}(\Xi)}
 \bigg(\int_{\R^2}|u|^2dx\bigg)^{\frac12} \bigg(\int_{\R^2}|\nabla u|^2dx\bigg)^{\frac12}+\frac1b  \int_{\R^2 }V(x)|u|^2dx \\
 &\leq\frac12\int_{\R^2}|u|^2dx +\frac12 \kappa^2_{\text{GN}}\text{meas}(\Xi)\int_{\R^2}|\nabla u|^2dx+\frac1b  \int_{\R^2 }V(x)|u|^2dx,
\end{align*}
where $\kappa_{\text{GN}}>0$ denotes the best constant associated with Gagliardo-Nirenberg inequality.
From this inequality, we obatin
\begin{align*}
  \int_{\R^2}|u|^2dx
 &\leq \kappa^2_{\text{GN}}\text{meas}(\Xi)\int_{\R^2}|\nabla u|^2dx+\frac2b  \int_{\R^2 }V(x)|u|^2dx
 \leq \max\bigg\{\kappa^2_{\text{GN}}\text{meas}(\Xi),  \frac2b\bigg\}\|u\|_E^2
\end{align*}
indicating that $\|u\|^2_{H^1(\R^2)}\leq C_{\Xi,b}\|u\|_E^2$, where $C_{\Xi,b}\triangleq\max\{1+\kappa^2_{\text{GN}}\text{meas}(\Xi), \frac{b+2}{b}\}>0$.
The proof of this lemma is finished.
\end{proof}

\begin{remark}\label{remark2.2}
With the help of Lemma \ref{imbedding}, for all $\lambda\geq1$, we can redefine the space $X_\lambda$ by
\[
X_\lambda \triangleq\bigg\{u\in E_\lambda:\int_{\R^2}\log(1+|x|)|u|^2dx<+\infty\bigg\}.
\]
\end{remark}

Finally, we consider the compact results for the nonlinearity $f^{R,\bar{\delta}}$ (defined in \eqref{fR}) which
would be very important in receiving the boundedness of $(C)$ sequence for $J_\lambda$.

\begin{lemma}\label{compact}
Suppose $(h_1)$ and $(h_3)$. Let $\{u_n\}\subset H^1(\R^2)$ be a sequence such that
$u_n\to u_0$ in $L^p(\R^2)$ and $u_n\to u_0$ a.e. in $\R^2$, then for all fixed $R>0$, passing to a subsequence if necessary,
\begin{equation}\label{compact1}
 \lim_{n\to\infty}\int_{\R^2}f^{R,\delta}(u_n)u_ndx= \int_{\R^2}f^{R,\delta}(u_0)u_0dx
 ~\text{and}~
 \lim_{n\to\infty}\int_{\R^2}F^{R,\delta}(u_n) dx= \int_{\R^2}F^{R,\delta}(u_0)u dx.
\end{equation}
If in addition we suppose that $\limsup\limits_{n\to\infty}\|u_n\|^2_{H^1(\R^2)}<  \frac{\pi}{\gamma+\alpha R^{\tau-2}}$,
then for every fixed $R>0$, passing to a subsequence if necessary,
\begin{equation}\label{compact2}
 \lim_{n\to\infty}\int_{\R^2}f^{R,2}(u_n)u_ndx= \int_{\R^2}f^{R,2}(u_0)u_0dx
 ~\text{and}~
 \lim_{n\to\infty}\int_{\R^2}F^{R,2}(u_n) dx= \int_{\R^2}F^{R,2}(u_0)u dx.
\end{equation}
\end{lemma}

\begin{proof}
It suffices to verify the first formulas in \eqref{compact1} and \eqref{compact2}, respectively.
Firstly, we study \eqref{compact1}. In this situation, without loss of generality, we can suppose that
$\sup\limits_{n\in N}\|u_n\|^2_{H^1(\R^2)}<\Sigma(R)$, where the constant $\Sigma(R)>0$
is dependent of $R>0$ and independent of $n\in \mathbb{N}$. Then, we claim that,
for every $\varepsilon>0$ independent of $R$, there is a constant $C_\varepsilon(R)>0$,
which may depend on $R$, such that
\begin{equation}\label{compact3}
  |f^{R,\delta}(t)|\leq \varepsilon|t|+ {C}_\varepsilon (R)|t|^{q-1}[ e^{(\Sigma(R)+1)^{-1} |t|^{2}}-1],~\forall t\in\R,
\end{equation}
where $q>1$ is arbitrary. To see it, via $(h_1)$,
\[
\lim_{t\to0}\frac{f^{R,\delta}(t)}{t}=\lim_{t\to0} e^{\alpha t^\tau} \lim_{t\to0}\frac{h(t)}{t}
=0~\text{uniformly in}~R>0.
\]
On the other hand, we need the following two facts: (1) For $t_0>0$, there are $b>a>0$ such that
\[
a e^{t^\delta}\leq be^{t^2},~\forall t\geq t_0;
\]
(2) For $t_0>0$, since $\lim\limits_{t\to+\infty}\frac{e^{t^2}}{e^{t^2}-1}=1$, there is a constant $C>0$ such that
\[
e^{t^2}\leq C(e^{t^2}-1),~\forall t\geq t_0.
\]
Thereby, for all $|t|\geq1$, by using $(h_3)$, there are $M_2(R)>M_1(R)>M$ such that
\begin{align*}
|f^{R,\delta}(t)|  & \leq M(e^{\gamma |t|^\delta}-1)e^{\alpha R^{\tau-\delta}|t|^{\delta}}\leq
Me^{(\gamma +\alpha R^{\tau-\delta})|t|^\delta}\leq M_1(R)e^{(\Sigma(R)+1)^{-1} |t|^{2}} \\
    & \leq M_2(R)[e^{(\Sigma(R)+1)^{-1} |t|^{2}}-1]\leq M_2(R)|t|^{q-1}[e^{(\Sigma(R)+1)^{-1} |t|^{2}}-1].
\end{align*}
So, we derive \eqref{compact3} by the above two formulas.
Combining \eqref{compact3} and \eqref{TM2}, we obtain that
\begin{equation*}
  \begin{gathered}
  \int_{\R^2} |f^{R,\delta}(u_n)u_n|dx \leq  \varepsilon\int_{\R^2}|u_n|^2dx+C_\varepsilon(R)
\int_{\R^2}|u_n|^{q}[ e^{(\Sigma(R)+1)^{-1} |u_n|^{2}}-1]dx\hfill\\
\ \ \ \  \leq  \varepsilon\int_{\R^2}|u_n|^2dx+C_\varepsilon(R)
\bigg(\int_{\R^2}|u_n|^{2q}dx\bigg)^{\frac12}\bigg(\int_{\R^2}[ e^{\frac{2\|u_n\|^2_{H^1(\R^2)}}{\Sigma(R)+1}
\big(|u_n|/\|u_n\|^2_{H^1(\R^2)}\big)^{2}}-1]dx\bigg)^{\frac12}\hfill\\
\ \ \ \ \leq \varepsilon\int_{\R^2}|u_n|^2dx+\bar{C}_\varepsilon(R)
\bigg(\int_{\R^2}|u_n|^{2q}dx\bigg)^{\frac12}. \hfill\\
  \end{gathered}
\end{equation*}
Thanks to $u_n\to u_0$ in $L^{2q}(\R^2)$ since $q>1$,
 we could exploit the generalized Lebesgue's Dominated Convergence theorem
 to get the desired result by $n\to\infty$ and then $\varepsilon\to0^+$.

 Next, we consider \eqref{compact2}. Similar to \eqref{compact3}, we claim that
 there is a ${C}^\prime_\varepsilon (R)>0$ such that
 \begin{equation}\label{compact4}
  |f^{R,2}(t)|\leq \varepsilon|t|+ {C}^\prime_\varepsilon (R)|t|^{q-1}[ e^{(\gamma+\alpha R^{\tau-2}) |t|^{2}}-1],~\forall t\in\R.
\end{equation}
Indeed, for all $|t|\geq1$, by using $(h_3)$, there are $\tilde{M} (R)>\bar{M}>M$ such that
\begin{align*}
|f^{R,2}(t)|  & \leq M(e^{\gamma |t|^\delta}-1)e^{\alpha R^{\tau-2}|t|^{2}}\leq M e^{\gamma |t|^\delta} e^{\alpha R^{\tau-2}|t|^{2}}
\leq   \bar{M} e^{\gamma |t|^2} e^{\alpha R^{\tau-2}|t|^{2}}\\
    & \leq\tilde{M}(R)[ e^{\gamma |t|^2}-1 ][e^{\alpha R^{\tau-2}|t|^{2}}-1]\leq \tilde{M}(R)
    [ e^{(\gamma+\alpha R^{\tau-2}) |t|^{2}}-1] \\
    & \leq\tilde{M}(R)|t|^{q-1}
    [ e^{(\gamma+\alpha R^{\tau-2}) |t|^{2}}-1]
\end{align*}
jointly with $f^{R,2}(t)=o(t)$ uniformly in $R>0$ as $t\to0$. As a consequence of
\eqref{compact4} and \eqref{TM2}, there holds
\begin{equation*}
  \begin{gathered}
  \int_{\R^2} |f^{R,2}(u_n)u_n|dx \leq  \varepsilon\int_{\R^2}|u_n|^2dx+C^\prime_\varepsilon(R)
\int_{\R^2}|u_n|^{q}[ e^{(\gamma+\alpha R^{\tau-2})|u_n|^{2}}-1]dx\hfill\\
\ \ \ \  \leq  \varepsilon\int_{\R^2}|u_n|^2dx+C^\prime_\varepsilon(R)
\bigg(\int_{\R^2}|u_n|^{2q}dx\bigg)^{\frac12}\bigg(\int_{\R^2}[ e^{ 2(\gamma+\alpha R^{\tau-2})\|u_n\|^2_{H^1(\R^2)}
\bigg(\frac{|u_n|}{\|u_n\|^2_{H^1(\R^2)}}\bigg)^{2}}-1]dx\bigg)^{\frac12}\hfill\\
\ \ \ \ \leq \varepsilon\int_{\R^2}|u_n|^2dx+\bar{C}^\prime_\varepsilon(R)
\bigg(\int_{\R^2}|u_n|^{2q}dx\bigg)^{\frac12}. \hfill\\
  \end{gathered}
\end{equation*}
Repeating the above arguments, we can verify the validity of \eqref{compact2}. So, we can accomplish the proof of this lemma.
\end{proof}

\begin{remark}
Thanks to the calculations in Lemma \ref{compact}, for all $u\in H^1(\R^2)$, we immediately derive that
\begin{equation}\label{compact5}
  \int_{\R^2}|f^{R,\bar{\delta}}(u )u|dx\leq \varepsilon\int_{\R^2}|u|^2dx+C_\varepsilon(R)
  \bigg(\int_{\R^2}|u|^{2q}dx\bigg)^{\frac12}
\end{equation}
provided $\|u\|^2_{H^1(\R^2)}$ is sufficiently small, where $q>1$ could be selected as required.
\end{remark}

\section{On problem \eqref{mainequation2}: the (sub)critical case}\label{Sec3}

In this section, we mainly make full use of the mountain-pass theorem
 introduced in \cite{MW,HL} to establish the existence of ground state solutions for
 Eq. \eqref{mainequation2}. Now,
we state the theorem
which is a consequence
of the Ekeland Variational Principle developed in \cite{Aubin}
 as follows.

\begin{proposition}\label{mp}
Let $Z$ be a Banach space and $\Psi\in \mathcal{C}^1(Z,\R)$ Gateaux
differentiable for all $v\in Z$, with G-derivative $\Phi^\prime(v)\in Z^{-1}$ continuous from the norm
topology of $Z$ to the weak $*$ topology of $Z^{-1}$ and $\Phi(0) = 0$. Let $S$ be a closed subset of $Z$ which
disconnects (archwise) $Z$. Let $v_0 = 0$ and $v_1\in Z$ be points belonging to distinct connected
components of $Z\backslash S$. Suppose that
\[
\inf_{S}\Phi\geq A>0~\text{and}~\Phi(v_1)\leq0
\]
and let $\Gamma=\{\gamma\in \mathcal{C}([0,1],Z):\gamma(0)~\text{and}~\gamma(1)=v_1\}$. Then
$$c=\inf_{\gamma\in\Gamma}\max_{t\in[0,1]}\Phi(\gamma(t))\geq A>0$$ and there is a $(C)_c$
sequence for $\Phi$.
\end{proposition}

In the sequel, if not specified conversely, we shall always assume $(V_1)-(V_3)$ and
$(h_1)-(h_3)$ just for simplicity.

\begin{lemma}\label{geometry}
Let $R>0$ be fixed, then for all $\lambda\geq1$, we have the conclusions:
\begin{itemize}
  \item There exist $A, \rho > 0$ independent of $\lambda\geq1$ such that $J_\lambda^{R,\bar{\delta}}(u)\geq A$ for all $\|u\|_{E_\lambda}=\rho$;
  \item There exists $\bar{u}\in X_\lambda$ such that $J_\lambda^{R,\bar{\delta}}(\bar{u})\leq0$ with $\|\bar{u}\|_{E_\lambda}>\rho$.
\end{itemize}
\end{lemma}

\begin{proof}
For all $u\in X_\lambda$, it follows from \eqref{v2} and \eqref{compact5} with $\|u\|_{\lambda}$ small enough,
\begin{align}\label{mp0}
 \nonumber  J_\lambda^{R,\bar{\delta}}(u)
 \nonumber   & =\frac{1}{2}\int_{\R^2}[|\nabla u|^2+\lambda V(x)|u|^2]dx+\frac14V_0(u)
- \int_{\R^2} F^{R,\bar{\delta}}(u)dx \\
 \nonumber   &  \geq \frac12\|u\|^2_{E_\lambda}-\frac{K_0}{4}|u|_{\frac83}^3-\varepsilon|u|_2^2-C_\varepsilon(R)
  |u|^{q}_{2q}\\
  &\geq \bigg( \frac12-\varepsilon C_0\bigg)\|u\|^2_{E_\lambda}-C_1\|u\|^4_{E_\lambda}-C_2\|u\|^q_{E_\lambda},
\end{align}
where $C_i>0$ with $i\in\{1,2\}$ is dependent of $R$ and the imbedding constant for $E_\lambda\hookrightarrow L^p(\R^2)$ for each $2\leq p<+\infty$
by Lemma \ref{imbedding}. Choosing $\varepsilon>0$ sufficiently small and $q>4$ in \eqref{compact5},
there exist $A>0$ and $\rho>0$ which are independent of $\lambda\geq1$ such that
$J_\lambda^{R,\bar{\delta}}|_{S_\lambda}\geq A$, where $S_\lambda\triangleq \{u\in X_\lambda:\|u\|_{E_\lambda}=\rho\}$.
The closed set $S_\lambda$ disconnects $X_\lambda$ in the two arcwise connected components
\[
X_\lambda^1\triangleq\{u\in X_\lambda:\|u\|_{E_\lambda}<\rho\}~\text{and}~
X_\lambda^2\triangleq\{u\in X_\lambda:\|u\|_{E_\lambda}>\rho\}.
\]
Furthermore, we can conclude that $0\in X_\lambda^1$ and there exists a $\bar{u}\in X_\lambda^2$
such that $J^{R,\bar{\delta}}_\lambda(\bar{u})\leq0$, because $\lim\limits_{t\to\infty}J^{R,\bar{\delta}}_\lambda(tu)=-\infty$
for all $u\in X_\lambda\backslash\{0\}$. Then, the value
\begin{equation}\label{mp1}
c_\lambda^{R,\bar{\delta}}\triangleq\inf_{\gamma\in\Gamma}\max_{t\in[0,1]}J_\lambda^{R,\bar{\delta}}(\gamma(t))\geq A>0
\end{equation}
determined by Propositions \ref{mp} is well-defined,
where $\Gamma^{R,\bar{\delta}}_\lambda=
\{\gamma\in \mathcal{C}([0,1],X_\lambda):\gamma(0)=0,\gamma(1)\in X_\lambda^2 ~\text{and}~J^{R,\bar{\delta}}_\lambda(\gamma(1))\leq0\}$.
The proof is completed.
\end{proof}

To characterize the mountain-pass level $c_\lambda^{R,\bar{\delta}}$ specifically,
we have the following lemma:

\begin{lemma}\label{unique}
Let $R>0$ be fixed and $\lambda\geq1$, then for every $u\in X_\lambda\backslash\{0\}$,
there exists a unique $t_u>0$ such that $t_uu\in \mathcal{N}_\lambda^{R,\bar{\delta}}$.
 Moreover,
the maximum of $J_\lambda^{R,\bar{\delta}}(tu)$ for $t\geq0$ is achieved at $t=t_u$.
\end{lemma}

\begin{proof}
According to the definition of $f^{R,\bar{\delta}}$ defined by \eqref{fR}, due to $(f_2)$, it simply concludes that
 \begin{equation}\label{ff}
~\text{The function}~f^{R,\bar{\delta}}(t)/t^3~\text{is increasing on}~t\in\R^+.
 \end{equation}
 Let $u\in X_\lambda\backslash\{0\}$ be fixed and define the function
 $\zeta(t)\triangleq J_\lambda^{R,\bar{\delta}}(tu)$ for $t\geq0$.
It would be easily noticed that $\zeta^\prime(t)\triangleq (J_\lambda^{R,\bar{\delta}})^\prime(tu)[u]$
 if and only if $t_uu\in \mathcal{N}_\lambda^{R,\bar{\delta}}$.
   Moreover, $\zeta^\prime(t)=0$ is equivalent to
   \[
   \frac{1}{t^2}\int_{\R^2}[|\nabla u|^2+\lambda V(x)|u|^2]dx+V_0(u)=\int_{\R^2}\frac{f^{R,\bar{\delta}}(tu)}{(tu)^3}u^4dx.
   \]
By using \eqref{ff}, the functions at the left and right hand sides are decreasing and increasing, respectively.
  Moreover, we compute that
  \[
  \zeta(t)=\frac1{t^4}\bigg(\frac{t^2}{2}\int_{\R^2}[|\nabla u|^2+\lambda V(x)|u|^2]dx+\frac14V_0(u)
- \int_{\R^2} F^{R,\bar{\delta}}(tu)dx\bigg).
  \]
Using some very similar calculations in Lemma \ref{geometry},
$\zeta(t) > 0$ for $t > 0$ small. Moreover, $\zeta(0) = 0$ and
$\zeta(t)=J_\lambda^{R,\bar{\delta}}(tu)<0$ for $t>0$ large. Therefore, from the previous conclusions, there exists a unique
$t_u> 0$ such that $\zeta^\prime(t_u)=0$, that is, $t_uu\in \mathcal{N}_\lambda^{R,\bar{\delta}}$. Furthermore,
$\zeta(t_u) = \max_{t\geq0}\zeta(t)$.
\end{proof}

Next, we consider the number below
\begin{equation}\label{mp2}
 d_\lambda^{R,\bar{\delta}}\triangleq\inf_{u\in X_\lambda\backslash\{0\}}\max_{t\geq0}J^{R,\bar{\delta}}_\lambda(tu).
\end{equation}
We derive the characterization with respect to $c_\lambda^{R,\bar{\delta}}$.

\begin{lemma}\label{characterization}
Let $R>0$ be fixed, then for every $\lambda\geq1$, there holds $m^{R,\bar{\delta}}_\lambda=c_\lambda^{R,\bar{\delta}}=d_\lambda^{R,\bar{\delta}}$,
where $m^{R,\bar{\delta}}_\lambda$ and $c^{R,\bar{\delta}}_\lambda$ are defined by \eqref{Nehari} and \eqref{mp1}, respectively.
\end{lemma}

\begin{proof}
The preceding lemma implies that $m^{R,\bar{\delta}}_\lambda=d_\lambda^{R,\bar{\delta}}$. Since
$\zeta(t)=J_\lambda^{R,\bar{\delta}}(t_0u)<0$ for $u\in X_\lambda\backslash\{0\}$ and $t_0$
large, define $\gamma_\lambda^{R,\bar{\delta}}[0,1]\to X_\lambda$ by
$\gamma_\lambda^{R,\bar{\delta}}(t)=tt_0u$, it follows that $\gamma_\lambda^{R,\bar{\delta}}\in \Gamma_\lambda^{R,\bar{\delta}}$
and, consequently, $c_\lambda^{R,\bar{\delta}}\leq d_\lambda^{R,\bar{\delta}}$.
Next, we show that $m^{R,\bar{\delta}}_\lambda\leq c_\lambda^{R,\bar{\delta}}$.
To end it, it suffices to prove that the manifold $\mathcal{N}_\lambda^{R,\bar{\delta}}$ separates
$X_\lambda$ into two components. Proceeding as in Lemma \ref{geometry}, we have
\[
(J_\lambda^{R,\bar{\delta}})^\prime(u)[u]
 \geq \bigg( \frac12-\varepsilon \bar{C}_0\bigg)\|u\|^2_{E_\lambda}-\bar{C}_1\|u\|^4_{E_\lambda}-\bar{C}_2\|u\|^q_{E_\lambda}
\]
with $\|u\|_{E_\lambda}$ small and $q>4$. So, there exists $\bar{\rho}>0$ independent of $\lambda\geq1$,
such that $(J_\lambda^{R,\bar{\delta}})^\prime(u)[u]>0$ when $0 < \|u\|_{E_\lambda} < \bar{\rho}$.
This proves that the component containing the origin
also contains a small ball around the origin. Moreover, $J_\lambda^{R,\bar{\delta}}(u)\geq0$
for all $u$ in this component, because $(J_\lambda^{R,\bar{\delta}})^\prime(tu)[u]\geq0$
for all $0\leq t\leq t_u$.  Thus, $\gamma_\lambda^{R,\bar{\delta}}(0)=0$ and $\gamma_\lambda^{R,\bar{\delta}}(1)$ are in different components,
which indicates that every path $\gamma_\lambda^{R,\bar{\delta}}\in  \Gamma_\lambda^{R,\bar{\delta}}$
has to cross $\mathcal{N}_\lambda^{R,\bar{\delta}}$. Therefore, we must have $m^{R,\bar{\delta}}_\lambda\leq c_\lambda^{R,\bar{\delta}}$
and $d^{R,\bar{\delta}}_\lambda\leq c_\lambda^{R,\bar{\delta}}$ and the lemma is proved.
\end{proof}

Now, let us turn to concentrate ourself on the properties
for the $(C)$ sequence at the mountain-pass level $c_\lambda^{R,\bar{\delta}}$.

\begin{lemma}\label{mplevel}
Let $R>0$ be fixed and for all $\lambda\geq1$,
there is a constant $c^{R,\delta}>0$ independent of $\lambda\geq1$
 such that $c_\lambda^{R,\delta}\leq c^{R,\delta}$ for every $\lambda\geq1$.
 Moreover, if we suppose $(h_4)$ additionally, then there exists a $\xi_0=\xi_0(R)>0$
 such that for all $\xi>{\xi}_0$
 \[
 c_\lambda^{R,2}<\frac{\pi}{4C_{\Xi,b}(1+\gamma+\alpha R^{\tau-2})},~\forall \lambda\geq1,
 \]
 where $\gamma>2$ and $C_{\Xi,b}>0$ come from $(h_3)$
 and Lemma \ref{imbedding}, respectively.
\end{lemma}

\begin{proof}
Let $\psi\in C_0^\infty(\Omega)$ satisfy $0\leq\psi\leq1$.
Since $V(x)\equiv0$ for all $x\in \Omega$ by $(V_2)$, then
\begin{align*}
J_\lambda^{R, {\delta}}(t\psi)
  & =\frac{t^2}{2}\int_{\Omega} |\nabla \psi|^2 dx+\frac{t^4}4V_0(\psi)
- \int_{\Omega} F^{R, {\delta}}(t\psi)dx.
\end{align*}
In view of the proof of Lemma \ref{unique},
$J_\lambda^{R, {\delta}}(t\psi)>0$ for $t>0$ small and $J_\lambda^{R, {\delta}}(t\psi)\leq0$ for $t>0$ large
and so the maximum of $J_\lambda^{R, {\delta}}(t\psi)$ for $t\geq0$ is achieved at $t=t_\psi>0$, namely
$\max\limits_{t\geq0}J_\lambda^{R, {\delta}}(t\psi)= J_\lambda^{R, {\delta}}(t_\psi\psi)$. We
choose $c^{R,\delta}= J_\lambda^{R, {\delta}}(t_\psi\psi)$ independent of $\lambda\geq1$ and get the desired result by Lemma \ref{mp2}.

Without loss of generality, we could suppose that $0\in \Omega$. Since $\Omega$
is an open set, there is a constant $\varrho>0$ such that $B_\varrho(0)\subset\Omega$.
Let us assume that $\rho=1$ just for the convenience of calculations. Now, we shall choose
  a $\varphi_0\in C_0^\infty(B_1(0))$ satisfying $0\leq\varphi_0\leq1$;
$\varphi_0(x)\equiv 1$
if $|x|\leq1/2$; $\varphi_0(x)\equiv 0$
if $|x|\geq1$; and $|\nabla \varphi_0|\leq1$
for all $x\in\R^2$.
 Recalling the definition of $f^{R,2}$ and $(h_4)$,
  $F^{R,2}(t)\geq\xi t^p$ with $p>4$ for all $t\in[0,1]$.
Thus, using Lemma \ref{zz}-(ii),
\begin{align}\label{mplevel1}
\nonumber J_\lambda^{R,2}(\varphi_0) &\leq \frac12\int_{B_1(0)} |\nabla \varphi_0|^2 dx+\frac12\int_{B_1(0)} |\varphi_0|^2 dx
\int_{B_1(0)}\log(1+|x|) |\varphi_0|^2 dx
- \int_{B_1(0)} F^{R,2}(\varphi_0)dx \\
      &< \frac12(1+ \pi\log2)\pi-{\xi}   \int_{B_{1/2}(0)}|\varphi_0|^{p} dx
    \leq \frac12(1+ \pi\log2)\pi-  \frac{{\xi}_1}{4} \pi =0.
\end{align}
where $\xi_1=2(1+ \pi\log2)$.
In particular, invoking from \eqref{mplevel1} that
 \begin{equation}\label{mplevel2}
\frac12\int_{B_1(0)}[|\nabla \varphi_0|^2+\lambda V(x)| \varphi_0|^2]dx+\frac14V_0(\varphi_0)
<  {\xi}_1 \int_{B_{1/2}(0)}|\varphi_0|^{p} dx.
 \end{equation}
Defining $\gamma_0^{R,2}(t)=t\varphi_0$, one deduces that $\gamma_0^{R,2}\in\Gamma^{R,2}_\lambda=\{\gamma\in \mathcal{C}([0,1],X_\lambda):
\gamma(0)=0, {J}^{R,2}_\lambda(\gamma(1))<0\}$ by \eqref{mplevel2}.
Therefore, we have that
\begin{align*}
 \max_{t\in[0,1]} {J}^{R,2}_\lambda(t\varphi_0)&\leq \max_{t\in[0,1]} \bigg \{\frac{t^2}{2}(1+ \pi\log2)\pi
-{\xi}t^p   \int_{B_{1/2}(0)}|\varphi_0|^{p} dx \bigg\}\\
  &\leq  \max_{t\geq0} \bigg \{\frac{t^2}{2}(1+ \pi\log2)
-\frac{\pi{\xi}}{4}t^p \bigg\}=\frac{\pi(p-1)(1+ \pi\log2)}{2p}\bigg[\frac{2(1+ \pi\log2)}{p\xi}\bigg]^{\frac2{p-2}}
\end{align*}
As a consequence, we can let the constant $\xi_0=\xi_0(R)$ be as follows
\[
\xi_0=\max\bigg\{\xi_1,
\frac{2(1+ \pi\log2)}{p} \bigg[\frac{4C_{\Xi,b}(p-1)(1+ \pi\log2)(1+\gamma+\alpha R^{\tau-2})}{2p}\bigg]^{\frac{p-2}2}\bigg\}.
\]
According to the definition of $c^{R,2}_\lambda$, we conclude it.
So, we can accomplish the proof of this lemma.
\end{proof}

With Lemma \ref{mplevel} in hand,
we start verifying the boundedness of $(C)$ sequence at the level $c^{R,\bar{\delta}}_\lambda$
 of $J_\lambda^{R,\bar{\delta}}$. Specifically,
 if $\{u_n\}\subset X_\lambda$ is a $(C)_{c^{R,\bar{\delta}}_\lambda}$ sequence of
$J_\lambda^{R,\bar{\delta}}$, we aim at showing $\|u_n\|_{X_\lambda}$
is uniformly bounded in $n\in \mathbb{N}$ for all $\lambda\geq1$ if $R>0$
is fixed. Before proceeding it,
we introduce the following two lemmas developed by P. L. Lions \cite{Lions}.

\begin{lemma}\label{Lionslemma}
Let $\{\rho_n\}\subset L^1(\R^2)$ be a bounded sequence and $\rho_n\geq0$, then there exists a subsequence, still denoted
by $\rho_n$, such that one of the following two possibilities occurs:
\begin{itemize}
  \item[\emph{(i)}] \emph{(Vanishing)} $\lim\limits_{n\to\infty}\sup\limits_{y\in\R^2}\int_{B_\varrho(y)}\rho_ndx=0$
for all $\varrho>0$;
  \item[\emph{(ii)}] \emph{(Non-Vanishing)} there are $\beta> 0$ and $\varrho <+\infty$ such that
$$\lim_{n\to\infty}\sup_{y\in\R^2}\int_{B_\varrho(y)}\rho_ndx=\beta.$$
\end{itemize}
\end{lemma}

\begin{lemma}\label{Vanish}
 Suppose that
$\{u_n\}$ is bounded in $L^2(\R^2)$ and $\{|\nabla u_n|\}$ is bounded in $L^2(\R^2)$
as well as
$$\lim_{n\to\infty}\sup_{y\in\R^2}\int_{B_\varrho(y)}|u_n|^2dx=0.$$
Then $u_n\to0 $ in $L^s
(\R^2)$ for $s\in(2,+\infty)$.
\end{lemma}

 \begin{lemma}\label{Vanishing}
 Let $R>0$ be fixed and for all $\lambda\geq1$,
suppose that $\{u_n\}\subset X_\lambda$ is a
$(C)$ sequence at the level $c^{R,\bar{\delta}}_\lambda$
 of $J_\lambda^{R,\bar{\delta}}$, if $(h_4)$ is additionally satisfied for $\bar{\delta}=2$,
 then $\{\|u_n\|_{E_\lambda}\}$ is uniformly bounded in $n\in \mathbb{N}$.
 Moreover, for all $\varrho>0$,
 \begin{equation}\label{Vanishing1}
\lim_{n\to\infty}\sup_{y\in\R^2}\int_{B_\varrho(y)}|u_n|^2dx=0
 \end{equation}
 could never occur.
 \end{lemma}

 \begin{proof}
 It follows from \eqref{ff} that
 \[
 f^{R,\bar{\delta}}(t)t-4F^{R,\bar{\delta}}(t)\geq0,~\forall t\in\R^+.
 \]
 From this inequality, we obtain that
 \begin{equation}\label{Vanishing2}
 c_\lambda^{R,\bar{\delta}}+o_n(1)\geq J_\lambda^{R,\bar{\delta}}(u_n)-\frac14(
 J_\lambda^{R,\bar{\delta}} )^\prime(u_n)[u_n]\geq \frac14\|u_n\|^2_{E_\lambda}
 \end{equation}
showing the first part of this lemma.
In view of Lemma \ref{imbedding}, both $\{|u_n|_r\}$ and $\{|\nabla u_n|_2\}$
 are uniformly bounded in $n\in \mathbb{N}$ for some $r>2$.
 Suppose by contradiction, we suppose that \eqref{Vanishing1} holds true.
 Thus, $u_n\to0 $ in $L^s
(\R^2)$ for all $s\in(2,+\infty)$ by Lemma \ref{Vanish},
which together with \eqref{Vanishing2} as well as Lemmas \ref{imbedding} and \ref{mplevel},
we have obtained that $\{u_n\}\subset H^1(\R^2)$ satisfies the assumptions in Lemma \ref{compact}.
Hence,
\[
 \|u_n\|^2_{E_\lambda}+V_1(u_n)=V_2(u_n)+\int_{\R^2}f ^{R,\bar{\delta}}(u_n)u_ndx+o_n(1)
 =o_n(1)
\]
where we have used Lemma \ref{zz}-(v) and \eqref{compact2}.
By the positivity of $V_1$, there holds
\[
\|u_n\|^2_{E_\lambda}=o_n(1)~\text{and}~V_1(u_n)=o_n(1)
\]
jointly with Lemma \ref{zz}-(v) and \eqref{compact1} again indicate that
\[
c_\lambda^{R,\bar{\delta}}=\frac12\|u_n\|^2_{E_\lambda}+\frac14V_1(u_n)-\frac14V_2(u_n)-
\int_{\R^2}F ^{R,\bar{\delta}}(u_n) dx+o_n(1)=o_n(1).
\]
Thereby, we arrive at a contradiction by \eqref{mp1}. The proof is completed.
 \end{proof}

 Thanks to Lemma \ref{Lionslemma}, with the help of Lemma \ref{Vanishing},
 we derive the following result which is crucial to prove that
 the sequence $\{u_n\}\subset X_\lambda$ is uniformly bounded in $X_\lambda$.

 \begin{lemma}\label{Non-Vanishing}
 Under the assumptions in Lemma \ref{Vanishing}, there exists a constant
 $\beta_0>0$, independent of $\lambda\geq1$, such that
 $$\lim_{n\to\infty}\sup_{y\in\R^2}\int_{B_\varrho(y)}|u_n|^2dx=\beta_0.$$
 \end{lemma}

 \begin{proof}
 Let $\rho_n=|u_n|^2\in L^1(\R^2)$, we could know that only the \textbf{Non-Vanishing}
 in Lemma \ref{Lionslemma} occurs because of Lemma \ref{Vanishing}. Then,
we divide the proof into intermediate steps.

\medspace
{\sc Step 1:} There exists a constant
 $\beta_\lambda=\beta(\lambda)>0$ such that
 $$\lim_{n\to\infty}\sup_{y\in\R^2}\int_{B_\varrho(y)}|u_n|^2dx=\beta_\lambda.$$

\medspace
Suppose, by contradiction, that $u_n\to0 $ in $L^s
(\R^2)$ for $s\in(2,+\infty)$.
It is very similar to the proof of Lemma \ref{Vanishing}, one derives a contradiction.

{\sc Step 2:} Conclusion.

 Suppose by contradiction that the uniform control from below of $L^2(\R^2)$-norm is false.
So, for any $k\in \mathbb{N}$, $k\neq 0$, there exist $\la_k>1$ and a $(C)_{c^{R,\bar{\delta}}_{\la_k}}$
sequence $\{u_{k,n}\}$ of $J_\lambda^{R,\bar{\delta}}$ such that
\[
|u_{k,n}|_2<\frac 1k, ~\text{definitely}.
\]
Then, by a diagonalization argument, for any $k\ge 1$, we can find an increasing sequence $\{n_k\}$ in $\mathbb{N}$
and $u_{n_k}\in X_{\la_{n_k}}$ such that
\[
J_{\la_{n_k}}^{R,\bar{\delta}} (u_{n_k})=c^{R,\bar{\delta}}_{\la_{n_k}}+o_k(1),
~
(1+\|u_{n_k}\|_{X_{\la_{n_k}}})\|(J_{\la_{n_k}}^{R,\bar{\delta}})'(u_{n_k})\|_{X_{\la_{n_k}}^{-1}}=o_k(1),
~
|u_{n_k}|_2=o_k(1),
\]
where $o_k(1)$ is a positive quantity which goes to zero as $k\to +\infty$.
 In this situation, we can repeat the proof of Lemma \ref{Vanishing}
 to reach a contradiction, again. The proof of this lemma is finished.
 \end{proof}

 Now, we can prove that
 the sequence $\{u_n\}\subset X_\lambda$ in Lemma \ref{Vanishing} is uniformly bounded in $X_\lambda$
 for some sufficiently large $\la>0$.

 \begin{lemma}\label{bounded}
 Let $R>0$ be fixed and
suppose that $\{u_n\}\subset X_\lambda$ is a
$(C)$ sequence at the level $c^{R,\bar{\delta}}_\lambda$
 of $J_\lambda^{R,\bar{\delta}}$. Moreover, we shall suppose $(h_4)$ additionally
 whence $\bar{\delta}=2$. Then, there is a $\lambda_0=\lambda_0(R)>1$ ($\lambda_0^\prime=\lambda_0^\prime(R)>1$
 for $\bar{\delta}=2$) such that the sequence
 $\{\|u_n\|_{X_\lambda}\}$ is uniformly bounded in $n\in \mathbb{N}$
and $\lambda>\lambda_0$ (or $\lambda>\lambda_0'$ for $\bar{\delta}=2$).
\end{lemma}

\begin{proof}
Combining Lemmas \ref{Vanishing} and \ref{Non-Vanishing},
there exists a constant
 $\beta_0>0$, independent of $\lambda\geq1$, such that
 $$\lim_{n\to\infty}\sup_{y\in\R^2}\int_{B_1(y)}|u_n|^2dx=\beta_0,$$
 where we have supposed that $\varrho=1$ in Lemma \ref{Non-Vanishing}.
 Up to a subsequence if necessary, there exists a sequence $\{y_n\}\subset\R^2$ such that
 \begin{equation}\label{bounded1}
\int_{B_1(y_n)}|u_n|^2dx=\frac{1}{2}\beta_0.
 \end{equation}
We claim that $\{y_n\}$ is uniformly bounded in $n\in \mathbb{N}$.
 Otherwise, we could suppose that $|y_n|\to\infty$ in the sense of a subsequence.
 Define
 \[
 \Xi_n^1\triangleq\{x\in B_1(y_n):V(x)<b\}~\text{and}~
  \Xi_n^2\triangleq\{x\in B_1(y_n):V(x)\geq b\}.
 \]
Since the set $\Xi\triangleq\{x\in \R^2
: V(x) <b\}$ is nonempty and has finite measure, one concludes that
 \begin{equation}\label{bounded2}
 \text{meas}(\Xi_n^1)\leq  ~\text{meas}(\{x\in\R^2:|x|\geq |y_n|-2,V(x)<b\})\to0~\text{as}~n\to\infty.
  \end{equation}
  In view of Lemma \ref{Vanishing}, $|u_n|_r$ with $r>2$ is uniformly bounded in $n\in \mathbb{N}$, then using \eqref{bounded2},
  \[
  \int_{\Xi_n^1}|u_n|^2dx\leq [\text{meas}(\Xi_n^1)]^{\frac{r-2}r}|u_n|_r^2=o_n(1)
  \]
  which reveals that
  \[
  \int_{\Xi_n^2}|u_n|^2dx=\int_{B_1(y_n)}|u_n|^2dx -\int_{\Xi_n^1}|u_n|^2dx= \frac{1}{2}\beta_0+o_n(1).
  \]
Thanks to $V(x)\geq0$ for all $x\in\R^2$ by $(V_1)$, using the definition of $\Xi_n^2$,
 \begin{equation}\label{bounded3}
 \int_{\R^2}V(x)|u_n|^2dx\geq \int_{\Xi_n^2}V(x)|u_n|^2dx\geq b\int_{\Xi_n^2} |u_n|^2dx =  \frac{1}{2}b\beta_0+o_n(1)
  \end{equation}
  Besides, in view of the proof of Lemma \ref{Vanishing} again, we have that
   \begin{equation}\label{bounded4}
  \{V_2(u_n)\} ~\text{and}~\bigg\{ \int_{\R^2}F ^{R,\bar{\delta}}(u_n) dx \bigg\}
  ~\text{are uniformly bounded in}~n\in \mathbb{N}~\text{and}~\lambda\geq1.
     \end{equation}
     So, combining \eqref{bounded3} and \eqref{bounded4}, we derive
     \begin{align}\label{bounded5}
    c_{\lambda}^{R,\bar{\delta}} & \geq \frac12\int_{\R^2}\lambda V(x)|u_n|^2dx-\frac14V_2(u_n)- \int_{\R^2}F ^{R,\bar{\delta}}(u_n) dx+o_n(1)
       \geq \frac{\lambda b\beta_0}{4}-C+o_n(1)
     \end{align}
   where the positive constants $b,\beta_0$ and $C$ are independent of $\lambda\geq1$.
   Using Lemma \ref{mplevel}, there exists a sufficiently large $\lambda_0=\lambda_0(R)>1$ ($\lambda_0^\prime=\lambda_0^\prime(R)>1$
 for $\bar{\delta}=2$)
   such that \eqref{bounded5} is false provided $\la>\la_0$.
 Hence, the sequence $\{y_n\}\subset\R^2$ appearing in \eqref{bounded1} is uniformly bounded in $n\in \mathbb{N}$.

 Consequently, passing to a subsequently if necessary,
 we suppose that $y_n\to y_0$ in $\R^2$. Taking \eqref{bounded1} into account,
 there holds
 \begin{equation}\label{bounded6}
\int_{B_2(y_0)}|u_n|^2dx\geq\frac{1}{4}\beta_0>0.
 \end{equation}
 Next, we shall investigate that $|u_n|_*=(\int_{\R^2}\log(1+|x|)u_n^2dx)^{\frac12}$
 is uniformly bounded in $n\in \mathbb{N}$.
 Let us choose a constant $\delta>0$ large enough to satisfy
 $\delta>|y_0|+2$. Moreover, one has
 \[
 1+|x-y|\geq 1+\frac{|y|}2\geq\sqrt{1+|y|},~\forall x\in B_\delta(0),~\forall y\in \R^2\backslash B_{2\delta}(0).
 \]
 Due to this choice for $\delta$ implying that $B_2(y_0)\subset B_\delta(0)$, by means of \eqref{bounded6},
 \begin{align}\label{bounded7}
  \nonumber    V_1(u_n) &= \int_{\R^2}\bigg(\int_{\R^2}\log(1+|x-y|)u^2_n(x)dx\bigg)u^2_n(y)dy\\
  \nonumber      & \geq  \int_{\R^2\backslash B_{2\delta}(0)}\bigg(\int_{B_\delta(0)}\log(1+|x-y|)u^2_n(x)dx\bigg)u^2_n(y)dy\\
    \nonumber    &\geq \bigg(\int_{B_\delta(0)} u^2_n(x)dx\bigg)\bigg[\int_{\R^2\backslash B_{2\delta}(0)} \log\bigg( 1+\frac{|y|}2\bigg)u^2_n(y)dy \bigg]\\
  \nonumber     & \geq \frac{\beta_0}8\int_{\R^2\backslash B_{2\delta}(0)} \log ( 1+ |y|)u^2_n(y)dy=\frac{\beta_0}8
 \nonumber   \bigg(\|u_n\|_*^2-\int_{B_{2\delta}(0)} \log ( 1+ |y| )u^2_n(y)dy\bigg)\\
    &\geq \frac{\beta_0}8(\|u_n\|_*^2-\log(1+2\delta)|u_n|_2^2).
 \end{align}
 Since we have proved that $|u_n|_2$ is uniformly bounded in $n\in \mathbb{N}$ for all $\lambda>\lambda_0$,
 with \eqref{bounded7}, it suffices to show that $\{V_1(u_n)\}$ is uniformly bounded in $n\in \mathbb{N}$ for all $\lambda>\lambda_0$.
 In fact, adopting \eqref{bounded4},
 \[
 0\leq V_1(u_n)\leq 4c_\lambda^{R,\bar{\delta}}+V_2(u_n)+4\int_{\R^2}F ^{R,\bar{\delta}}(u_n) dx+o_n(1)
 \]
 finishing the proof of this lemma.
\end{proof}

We are in a position to present the proof of Theorem \ref{maintheorem2}.

\begin{proof}[\textbf{\emph{Proof of Theorem \ref{maintheorem2}}}]
Combining Proposition \ref{mp} and Lemma \ref{geometry}, for every fixed $R>0$,
 the variational functional $J_\lambda^{R,\bar{\delta}}$
admits a $(C)$ sequence $\{u_n\}\subset X_\lambda$ at the level $c_\lambda^{R,\bar{\delta}}$
for all $\lambda\geq1$. With the help of Lemma \ref{bounded},
$\{u_n\}$ is bounded in $X_\lambda$ whenever $\lambda\geq \lambda_0$,
where $\lambda_0>0$ depending on $R>0$ is determined by Lemma \ref{bounded}.
Passing to a subsequence if necessary, there is a $u\in X_\lambda$
such that $u_n\rightharpoonup u$ in $X_\lambda$,
$u_n\to u$ in $L^s(\R^2)$ for every $2\leq s<\infty$ by Lemma \ref{zz}-(i)
and $u_n\to u$ a.e. in $\R^2$ as $n\to\infty$. So, the remaining part is
to verify that $u_n\to u$ in $X_\lambda$ by \eqref{mp1} and Lemma \ref{characterization}.
Moreover $u\geq0$ by $(h_1)$ and we omit the details.

It follows from Lemma \ref{zz}-(ii) and (iii) as well as Lemma
 \ref{yy} that
 \begin{align}\label{2Proof1}
\nonumber  V_1^\prime(u_n)[u_n-u]& =4B_1(u_n^2,u_n(u_n-u))=4B_1(u_n^2, (u_n-u)^2)+4B_1(u_n^2,u(u_n-u)) \\
\nonumber    &=  4B_1(u_n^2, (u_n-u)^2)+o_n(1)\leq 8|u_n-u|_2\|u_n\|_*^2+o_n(1)\\
&=o_n(1).
 \end{align}
Using Lemma \ref{zz}-(i) and (iv),
\begin{align}\label{2Proof2}
\nonumber  |V_2^\prime(u_n)[u_n-u]|& =4|B_2(u_n^2,u_n(u_n-u))|\leq4|B_2(u_n^2, (u_n-u)^2)|+4|B_2(u_n^2,u(u_n-u))| \\
\nonumber    &= 4K_0|u_n^2|_{\frac43}| (u_n-u)^2|_{\frac43}  + 4K_0|u_n^2|_{\frac43}|u (u_n-u)|_{\frac43}  \\
\nonumber &\leq4K_0|u_n|^2_{\frac83}| u_n-u |^2_{\frac83}  +4K_0|u_n|^2_{\frac83}|u  | _{\frac83}| u_n-u | _{\frac83} \\
&=o_n(1).
 \end{align}
 Using the same arguments in Lemma \ref{Vanishing},
 one could verify $\{u_n\}$ satisfies all of the assumptions in Lemmas \ref{compact} and so
 \begin{equation}\label{2Proof3}
 \int_{\R^2}f^{R,\bar{\delta}}(u_n)(u_n-u)dx=o_n(1).
 \end{equation}
 As a consequence of \eqref{2Proof1}, \eqref{2Proof2} and \eqref{2Proof3}, we obtain
 \begin{align*}
  o_n(1) &= (J_\lambda^{R,\bar{\delta}})^\prime(u_n)[u_n-u] \\
    & =\int_{\R^2}\big[\nabla u_n\nabla (u_n-u)+\lambda V(x)u_n(u_n-u)\big]dx+ o_n(1)\\
    &=\|u_n\|_{E_\lambda}^2-\|u\|_{E_\lambda}^2+o_n(1)=\|u_n-u\|_{E_\lambda}^2 +o_n(1).
 \end{align*}
  Next, we get $\|u_n-u\|_*=o_n(1)$ and then we are done.
  Indeed, proceeding as for \eqref{bounded7}, one has
  \begin{align*}
    o_n(1) & =B_1(u_n^2, (u_n-u)^2)
  \geq \frac{\beta_0}8(\|u_n-u\|_*^2-\log(1+2\delta)|u_n-u|_2^2)
  \end{align*}
  yielding the desired result. The proof is completed.
\end{proof}

 \section{Proofs of Theorems \ref{maintheorem1} and \ref{maintheorem3}}\label{Linftyestimate}

 In this section, as our discussions in the Introduction,
 we must take the uniform $L^\infty$-estimate for the nontrivial solution
 obtained in Theorem \ref{maintheorem2}. Let $u_R\in X_\lambda$
 be a ground state solution associated with Eq. \eqref{mainequation2},
 according to the definition of $f^{R,\bar{\delta}}$ which is defined as in \eqref{definition2},
 then it is a ground state solution for Eq. \eqref{mainequation1}
 provided $|u_R|_\infty\leq R$. So, the key idea is to find a constant
 $C_0>0$ which is independent of $R>0$ satisfying $|u_R|_\infty\leq C_0$.
 Therefore, we define $R$ to equal to such a constant $C_0$. Have it in mind,
 we derive our direction in this section.

 To the aim. we must firstly prove that the constants $\lambda_0(R),\lambda_0^\prime(R)$
 and $\xi_0(R)$ appearing in Theorem \ref{maintheorem2} do not depend on $R$.
 Let us recall the two cases
in Theorem \ref{maintheorem2}, to proceed it clearly, we would split it
 by two subsections:
(I) $\bar{\delta}=\delta\in(0,2)$; (II) $\bar{\delta}=2$.

 In the Cases (I) and (II), we shall choose $\alpha^*=\frac{1}{R^{\tau-\delta}}>0$ and
  $\tau_*=2+\frac{1}{R}>0$, respectively.

 \subsection{The Case (I) in \eqref{definition2}$: \bar{\delta}=\delta\in(0,2)$}
In this Subsection, we shall suppose that $V$ satisfies $(V_1)-(V_3)$
and the nonlinearity $f$ defined in \eqref{form}
requires $(h_1)-(h_3)$.

With the choice of $\alpha^*=\frac{1}{R^{\tau-\delta}}>0$ in this subsection,
we improve \eqref{compact3} in the sense: for each $\varepsilon>0$, there is a constant
$C_\varepsilon>0$ independent of $R>0$ such that
\begin{equation}\label{compact3I}
 |f^{R,\delta}(t)|\leq \varepsilon|t|+ {C}_\varepsilon |t|^{q-1}[ e^{(K+1)^{-1} |t|^{2}}-1],~\forall t\in\R,
\end{equation}
where $q>1$ is arbitrary and $K>0$ independent of $R>0$ which is determined later.
Via the two facts in the proof of Lemma \ref{compact}, for all $\alpha\in(0,\alpha^*)$ and $|t|\geq1$,
there are $M_2\geq M_1\geq M$ independent of $R>0$ such that
\begin{align*}
|f^{R,\delta}(t)|  & \leq M(e^{\gamma |t|^\delta}-1)e^{\alpha^* R^{\tau-\delta}|t|^{\delta}}\leq
Me^{(\gamma +1)|t|^\delta}\leq M_1 e^{(K+1)^{-1} |t|^{2}} \\
    & \leq M_2 [e^{(K+1)^{-1} |t|^{2}}-1]\leq M_2 |t|^{q-1}[e^{(K+1)^{-1} |t|^{2}}-1]
\end{align*}
jointly with $f^{R,\delta}(t)=o(t)$ uniformly in $R$ as $t\to0$, we have \eqref{compact3I} at once.

 \begin{lemma}\label{Ia}
 If $\lambda\geq1$ and let $\{u_n\}\subset X_\lambda$
 be a $(C)$ sequence of $J_\lambda^{R,\delta}$ at the level
 $c_{\lambda}^{R,\delta}$, then $\{u_n\}$ is
   is uniformly bounded in $n\in \mathbb{N}$
and $R>0$, that is,
there is a constant $K>0$ independent of $n\in \mathbb{N}$ and $R>0$ such that
\begin{equation}\label{Ia1}
\sup_{n\in \mathbb{N}}\|u_n\|_{H^1(\R^2)} ^2< \frac{2\pi K}{q}<+\infty.
\end{equation}
\end{lemma}

\begin{proof}
We claim that there are constants $A_0>0$ and $c>0$ independent of $R>0$ and $\lambda\geq1$ such that
\begin{equation}\label{Ia2}
A_0<c_{\lambda}^{R,\delta}\leq c<+\infty,~\forall R>0~\text{and}~\lambda\geq1.
\end{equation}
Indeed, recalling the definition of $f^{R,\delta}$, one has that $F^{R,\delta}(t)\geq H(t)$
for all $t\in\R$ and so $J_\lambda^{R,\delta}(u)\geq I_\lambda(u)$ for all $u\in X_\lambda$, where the variational functional $I_\lambda:X_\lambda\to\R$ is defined by
\begin{equation}\label{Ia3}
I_\lambda(u)=\frac{1}{2}\int_{\R^2}[|\nabla u|^2+\lambda V(x)|u|^2]dx+\frac14V_0(u)-
\int_{\R^2} H(u)dx.
\end{equation}
Let us choose the constant $c_\lambda>0$ to be a mountain-pass level associated with $I_\lambda$,
the existence of such number follows Lemma \ref{geometry}.
In light of $\psi$ as in Lemma \ref{mplevel}, one has
 \[
 I_\lambda(t\psi)=\frac{t^2}{2}\int_{\Omega} |\nabla \psi|^2 dx+\frac{t^4}4V_0(\psi)-
\int_{\Omega} H(t\psi)dx.
 \]
 Then, using the same arguments in Lemma \ref{mplevel}, we could find a $c>0$ independent of
 $\lambda\geq1$ such that $c_\lambda\leq c$. Besides, by exploiting \eqref{compact3I},
 we proceed as \eqref{mp0} to derive such a constant $A_0>0$ satisfying $c_\lambda^{R,\delta}\geq A_0$.
So, \eqref{Ia2} holds true.
 Combining \eqref{Vanishing2} and \eqref{Ia2} as well as Lemma \ref{imbedding},
 we would receive the desired result \eqref{Ia1}. The proof is complete.
 \end{proof}

 \begin{remark}\label{PP1}
 In view of \eqref{bounded5}, due to \eqref{Ia2}, one can see that we can determine
 the constant $\lambda_0(R)$ to be independent of $R>0$. Moreover, by
using \eqref{compact3I} and \eqref{Ia1}, one would conclude that \eqref{compact1}
 holds independently with respect to $R>0$.
 \end{remark}

 \begin{lemma}\label{Ib}
Let $u_R\in X_\lambda$ be a nonnegative solution of Eq. \eqref{mainequation2} with $\bar{\delta}=\delta$
established by Theorem \ref{maintheorem2} for all fixed $R>0$, if
 $\alpha^*=\frac{1}{R^{\tau-\delta}}>0$, then for all $\alpha\in(0,\alpha^*)$ and $\tau>2$, we have
$$
|u_R|_\infty \leq C_0,  ~\forall R>0.
$$
for some $C_0>0$ independent of $R>0$ and $\lambda>\lambda_0(R)$.
\end{lemma}

\begin{proof}
In view of Section 3, we know that the nonnegative solution $u_R\neq0$
of Eq. \eqref{mainequation2} is established by looking
for the weak limit of $\{u_n\}\subset X_\lambda$ which is a $(C)$ sequence
of $J_\lambda^{R,\delta}$ at the level $c_\lambda^{R,\delta}$.
 So, combining the Fatou's lemma and \eqref{Ia1},
$\|u_R\|^2_{H^1(\R^2)}\leq \frac{2\pi K}{q}$
for all $R>0$, where $K>0$ is a constant independent of $R>0$. We claim that there is a
constant $C_1>0$ independent of $R>0$
\begin{equation}\label{Ib1}
|f^{R, \delta }(u_R) |_2\leq C_1<+\infty.
\end{equation}
In fact, due to \eqref{compact3I}, by applying \eqref{TM2}, then it suffices to show that
\[
\begin{gathered}
\int_{\R^2}|u_R|^{2(q-1)}[ e^{2(K+1)^{-1} |u_R|^{2}}-1]dx
\leq \bigg(\int_{\R^2}|u_R|^{2q}dx\bigg)^{\frac{q-1}{q}}\bigg(\int_{\R^2}
[ e^{2q(K+1)^{-1} |u_R|^{2}}-1]
dx\bigg)^{\frac{1}{q}}\hfill\\
\ \ \ \ =\bigg(\int_{\R^2}|u_R|^{2q}dx\bigg)^{\frac{q-1}{q}}\bigg(\int_{\R^2}
[ e^{\frac{2q\|u_R\|^2_{H^1(\R^2)}}{K+1} \big( |u_R|/\|u_R\|_{H^1(\R^2)}  \big)^{2}}-1]
dx\bigg)^{\frac{1}{q}}\leq C
\hfill\\
\end{gathered}
\]
for some $C>0$ independent of $R>0$.

Then, proceeding as the proof of \cite[Lemma 4.10]{AS0}, we shall show that
$|u_R|_\infty\leq C_0$, where $C_0>0$ is independent of $R>0$. Taking the convenience of the reader into account,
 it should be done in detail.
Since $u_R$ is a nonnegative solution of Eq. \eqref{mainequation2}, by $(V_1)$,
 $u_R$ must satisfy the inequality
$$
-\Delta u_R+u_R\leq [1+g_{u_R}(x)]
 u_R  + f^{R,\delta}(u_R),~x\in\R^2,
$$
where $g_{u_R}(x)$ is defined as in Lemma \ref{zz}. Set $\Upsilon\triangleq 1+|g_{u_R}|_\infty$,
then $\Upsilon\in(1,+\infty)$ does not depend on $R>0$ by Lemma \ref{zz} and \eqref{Ia1}.
From this and \eqref{Ib1}, there is a ${v}_R \in H^{2}(\R^2)$ such that
 $$
 -\Delta  {v}_R+ {v}_R=
 \Upsilon u_R+ f^{R,\delta}(u_R)~ \mbox{in} ~ \R^2.
 $$
Next, we fix the test function
 $$
z_r(x)=\phi(x/r)(u_R- {v}_R)^{+}(x) \in X_\lambda,
 $$
 where $\phi \in C_{0}^{\infty}(\R^2)$ satisfies
 $$
 0 \leq \phi(x) \leq 1 \quad \forall x \in \R^2, \quad \phi(x)=1 \quad \forall x \in B_1(0) \quad \mbox{and} \quad \phi(x)=0 \quad \forall x \in
\R^2\backslash B _{2}(0).
 $$
Using the function test $z_r$ on $-\Delta (u_R- {v}_R)+(u_R- {v}_R)\leq0$ in $\R^2$, we get the inequality below
 $$
 \int_{\R^2}[\nabla (u_R- {v}_R)\nabla z_r + (u_R- {v}_R)z_r]\,dx \leq 0, \quad \forall r>0.
 $$
Since
$$
z_r \to (u_R- {v}_R)^+ \quad \mbox{as} \quad r \to +\infty \quad \mbox{in} \quad H^1(\R^2),
$$
by the Lebesgue's Dominated Convergence theorem,
we arrive at
$$
 \int_{\R^2}|\nabla (u_R- {v}_R)^+|^2+|(u_R-v_R)^+|^2  \,dx \leq 0,
$$
implying that
$$
0 \leq u_R(x) \leq {v}_R(x), \quad \forall x \in \R^2.
$$
By using the continuous Sobolev embedding $H^2(\R^2) \hookrightarrow L^{\infty}(\R^2)$, there is $C_2>0$ independent of $R>0$ such that
$$
| {v}_R|_\infty \leq C_2\| {v}_R\|_{H^{2}(\R^2)}, \quad \forall R>0
$$
which together with the last fact gives that
$$
|u_R|_\infty \leq C_3\| {v}_R\|_{H^{2}(\R^2)}, \quad \forall R>0.
$$
On the other hand, by Br\'{e}zis \cite[Theorem 9.25]{Brezis}, there is $C_4>0$ independent of $R>0$ such that
$$
\| {v}_R\|_{H^{2}(\R^2)} \leq C_4|\Upsilon u_R+ f^{R,\delta}(u_R)|_2, \quad \forall R>0,
$$
from where and \eqref{Ia1} as well as \eqref{Ib1}, it follows that
$$
\| {v}_R\|_{H^{2}(\R^2)} \leq C_5, \quad \forall R>0
$$
for some $C_5>0$ independent of $R>0$. Have this in mind, we must have
$$
|u_R|_\infty \leq C_0, \quad \forall R>0.
$$
for some $C_0>0$ independent of $R>0$, showing the desired result.
 \end{proof}

 \subsection{The Case (II) in \eqref{definition2}$: \bar{\delta}=2$}
In this Subsection, we shall suppose that $V$ satisfies $(V_1)-(V_3)$
and the nonlinearity $f$ defined in \eqref{form}
requires $(h_1)-(h_3)$ as well as $(h_4)$.

With the choice of $\tau_*=2+\frac{1}{R}>0$ in this subsection,
we improve \eqref{compact4} in the sense: for all $\varepsilon>0$, there is a constant
$C_\varepsilon^\prime>0$ independent of $R>e$ such that
 \begin{equation}\label{compact4II}
  |f^{R,2}(t)|\leq \varepsilon|t|+ {C}^\prime_\varepsilon  |t|^{q-1}[ e^{(\gamma+\alpha e^{\frac1e}) |t|^{2}}-1],~\forall t\in\R.
\end{equation}
Firstly, one can observe that $\lim\limits_{R\to+\infty}R^{\frac{1}{R}}=1$
and the function $R^{\frac{1}{R}}$ is strictly decreasing in $R\in(e,+\infty)$,
then $0<R^{\frac{1}{R}}\leq e^{\frac{1}{e}}$ for each $R\in(e,+\infty)$.
For all $|t|\geq1$, by using $(h_3)$, there is ${M}_1>M$ independent of $R>e$ such that
\begin{align*}
|f^{R,2}(t)|  & \leq M(e^{\gamma |t|^\delta}-1)e^{\alpha R^{\tau^*-2}|t|^{2}}= M e^{\gamma |t|^\delta} e^{\alpha R^{\frac1R}|t|^{2}}
\leq  M e^{\gamma |t|^\delta} e^{\alpha e^{\frac1e}|t|^{2}}\\
    & \leq {M}_1[ e^{\gamma |t|^2}-1 ][e^{\alpha e^{\frac1e}|t|^{2}}-1]\leq  {M}_1
    [ e^{(\gamma+\alpha e^{\frac1e}) |t|^{2}}-1] \\
    & \leq M_1|t|^{q-1}
    [ e^{(\gamma+\alpha e^{\frac1e}) |t|^{2}}-1]
\end{align*}
which together with $f^{R,2}(t)=o(t)$ uniformly in $R>0$ as $t\to0$.

 \begin{lemma}\label{IIa}
 If $\lambda\geq1$, there are some constants $\bar{A}_0>0$ and $\xi_0>0$ independent of $R>e$ such that for all $\xi>\xi_0$,
 there holds
 \begin{equation}\label{IIa1}
 \bar{A}_0\leq c_\lambda^{R,2}<\frac{\pi}{4C_{\Xi,b}(1+\gamma+\alpha e^{\frac1e})},~\forall R>e.
 \end{equation}
\end{lemma}

\begin{proof}
Applying \eqref{compact4II} to \eqref{mp0}, one can find such a $\bar{A}_0>0$
and the details are left. By the definition of $f^{R,2}$, then
$c_\lambda^{R,2}\leq c_\lambda$, where $c_\lambda$ is a mountain-pass level corresponding to
the variational functional $I_\lambda$ defined by \eqref{Ia3}. Let $\varphi_0$
be as in Lemma \ref{mplevel}, then according to $(h_3)$ and Lemma \ref{zz}-(ii),
\begin{align}\label{IIa2}
\nonumber I_\lambda (\varphi_0) &\leq \frac12\int_{B_1(0)} |\nabla \varphi_0|^2 dx+\frac12\int_{B_1(0)} |\varphi_0|^2 dx
\int_{B_1(0)}\log(1+|x|) |\varphi_0|^2 dx
- \int_{B_1(0)} H(\varphi_0)dx \\
      &< \frac12(1+ \pi\log2)\pi-{\xi}   \int_{B_{1/2}(0)}|\varphi_0|^{p} dx
    \leq \frac12(1+ \pi\log2)\pi-  \frac{{\xi}_1}{4} \pi =0.
\end{align}
where $\xi_1=2(1+ \pi\log2)$.
In particular, invoking from \eqref{IIa2} that
 \begin{equation}\label{IIa3}
\frac12\int_{B_1(0)}[|\nabla \varphi_0|^2+\lambda V(x)| \varphi_0|^2]dx+\frac14V_0(\varphi_0)
<  {\xi}_1 \int_{B_{1/2}(0)}|\varphi_0|^{p} dx.
 \end{equation}
 Defining $\gamma_0 (t)=t\varphi_0$, one deduces that $\gamma_0 \in\Gamma _\lambda=\{\gamma\in \mathcal{C}([0,1],X_\lambda):
\gamma(0)=0,I_\lambda(\gamma(1))<0\}$ by \eqref{mplevel2}.
Therefore, we have that
\begin{align*}
 \max_{t\in[0,1]}I_\lambda(t)&\leq \max_{t\in[0,1]} \bigg \{\frac{t^2}{2}(1+ \pi\log2)\pi
-{\xi}t^p   \int_{B_{1/2}(0)}|\varphi_0|^{p} dx \bigg\}\\
  &\leq  \max_{t\geq0} \bigg \{\frac{t^2}{2}(1+ \pi\log2)
-\frac{\pi{\xi}}{4}t^p \bigg\}=\frac{\pi(p-1)(1+ \pi\log2)}{2p}\bigg[\frac{2(1+ \pi\log2)}{p\xi}\bigg]^{\frac2{p-2}}
\end{align*}
As a consequence, the constant $\xi_0$ independent of $R>e$ can be chosen as follows
\[
\xi_0=\max\bigg\{\xi_1,
\frac{2(1+ \pi\log2)}{p} \bigg[\frac{4C_{\Xi,b}(p-1)(1+ \pi\log2)(1+\gamma+\alpha e^{\frac1e})}{2p}\bigg]^{\frac{p-2}2}\bigg\}
\]
showing that $c_\lambda<\frac{\pi}{4C_{\Xi,b}(1+\gamma+\alpha e^{\frac1e})}$ for all $R>e$.
By the fact $c_\lambda^{R,2}\leq c_\lambda$ for every $R>0$,
  we would get the desired result immediately. The proof is completed.
\end{proof}

 \begin{remark}\label{PP2}
 With Lemma \ref{IIa} in hand, for all $\lambda\geq1$, we obtain that
 \begin{equation}\label{IIa4}
 \limsup_{n\to\infty}\|u_n\|_{H^1(\R^2)}^2 <\frac{\pi}{\gamma+\alpha e^{\frac1e}}~\forall R>e,
 \end{equation}
 where $\{u_n\}\subset X_\lambda$ is a $(C)$ sequence of $J_\lambda^{R,2}$ at the
 level $c_\lambda^{R,2}$. Indeed, using \eqref{Vanishing2} and \eqref{IIa1}
 together with Lemma \ref{imbedding}, it is obvious.
 With the help of \eqref{bounded5} and \eqref{IIa1}, we can also conclude that
 the constant $\lambda_0^\prime>0$ is independent of $R>e$.
 Moreover, it follows from
 \eqref{compact4II} and \eqref{IIa4} that \eqref{compact2}
 holds independently with respect to $R>e$.
 \end{remark}

 \begin{lemma}\label{IIb}
Let $u_R\in X_\lambda$ be a nonnegative solution of Eq. \eqref{mainequation2} with $\bar{\delta}=2$
established by Theorem \ref{maintheorem2} for all fixed $R>e$, if
 $\tau_*=2+\frac{1}{R }>0$, then for all $\alpha>0$ and $\tau\in[2,\tau_*)$, we have
$$
|u_R|_\infty \leq C_0^\prime,  ~\forall R>e.
$$
for some $C_0^\prime>0$ independent of $R>e$ and $\lambda>\lambda'_0(R)$.
\end{lemma}

\begin{proof}
In light of Lemma \ref{Ib}, it suffices to prove that
\begin{equation}\label{Ib2}
|f^{R, 2}(u_R) |_2\leq C_1^\prime
\end{equation}
for some $C_1^\prime >0$ independent of $R>e$.
Firstly, due to the Fatou's lemma,
$\|u_R\|^2_{H^1(\R^2)}<\frac{\pi}{\gamma+\alpha e^{\frac1e}}$ by \eqref{IIa4}
for all $R>e$. Hence, using \eqref{compact4II} with $q=4$, we exploit \eqref{TM2} to have
\begin{equation*}
\begin{gathered}
  \int_{\R^2}|f^{R,2}(u_R)|^2dx\leq \int_{\R^2}|u_R|^2dx
  + {C}^\prime  \int_{R^2} |u_R|^{2(q-1)}[ e^{(\gamma+\alpha e^{\frac1e}) |u_R|^{2}}-1]dx\hfill\\
  \ \ \ \ \leq \int_{\R^2}|u_R|^2dx
  + {C}^\prime  \bigg(\int_{R^2} |u_R|^{8}\bigg)^{\frac{3}{4}}
  \bigg(\int_{\R^2}[ e^{4(\gamma+\alpha e^{\frac1e})\|u_R\|_{H^1(\R^2)}^2( |u_R|/ \|u_R\|_{H^1(\R^2)})^{2}}-1]dx\bigg)^{\frac{1}{4}}\hfill\\
  \ \ \ \ \leq C_1^\prime\hfill\\
  \end{gathered}
\end{equation*}
as required.
 \end{proof}

Now, we can present the proof of Theorem \ref{maintheorem1} below.

\begin{proof}[\textbf{\emph{Proof of Theorem \ref{maintheorem1}}}]
Recalling Theorem \ref{maintheorem2}, we have established a nonnegative ground state solution for
Eq. \eqref{mainequation2} under the suitable assumptions.
Let us denote the obtained ground state solution by $u_R$.
Thanks to the explanations in Remarks \ref{PP1} and \ref{PP2},
the constants $\lambda_0$ for $\bar{\delta}=\delta\in(0,2)$,
$\lambda_0^\prime$ and $\xi_0$ for $\bar{\delta}=2$
are independent of $R>0$ and $R>e$, respectively.
It follows from Lemmas \ref{IIa} and \ref{IIb}
that we colud choose $R=C_0$ and $R=\max\{C_0^\prime,e\}$
for $\bar{\delta}=\delta\in(0,2)$ and $\bar{\delta}=2$, respectively.
In this situation, $\alpha^*=\frac{1}{C_0^{\tau-\delta}}$
and $\tau_*=2+\frac{1}{\max\{C_0^\prime,e\}}$.
So, $u_R$ is a nonnegative ground state solution of Eq. \eqref{mainequation1}.
The proof is completed.
 \end{proof}

 Finally, we are concerned with the asymptotical behavior
 of ground state solutions of Eq. \eqref{mainequation1} obtained in Theorem \ref{maintheorem1}
 as $\lambda\to+\infty$.

 Before showing the proof of Theorem \ref{maintheorem3}, via the same constant $R>0$
 determined in the proof of Theorem \ref{maintheorem2},
we need the variational functionals below
\[
\left\{
  \begin{array}{ll}
\displaystyle      J_\Omega(u)=\frac{1}{2}\int_{\Omega} |\nabla u|^2 dx+\frac14V_0|_{\Omega}(u)
- \int_{\Omega} F(u)dx, & \forall u\in H_0^1(\Omega),\\
\displaystyle       J_\Omega^{R,\bar{\delta}}(u)=\frac{1}{2}\int_{\Omega} |\nabla u|^2 dx+\frac14V_0|_{\Omega}(u)
- \int_{\Omega} F^{R,\bar{\delta}}(u)dx,  & \forall u\in H_0^1(\Omega),
  \end{array}
\right.
 \]
 where the functional $V_0|_{\Omega}:H^1_0(\Omega)\to\R$ which is defined by $V_0|_{\Omega}=V_1|_{\Omega}-V_2|_{\Omega}$ with
\[
V_1|_{\Omega}(u)\triangleq\int_{\Omega}\int_{\Omega}\log(1+|x-y|)u^2(x)u^2(y)dxdy,~ \forall u\in H^1_0(\Omega),
\]
and
\[
V_2|_{\Omega}(u)\triangleq\int_{\Omega}\int_{\Omega}\log\bigg(1+\frac{1}{|x-y|}\bigg)u^2(x)u^2(y)dxdy,~  \forall u\in H^1_0(\Omega),
\]
Since $\text{meas}(\Omega)<+\infty$, there is a constant $\varrho>0$ such that $\Omega\subset B_\varrho(0)$
and so $$0\leq\log(1+|x-y|)\leq \log(1+2\varrho),~\forall x,y\in\Omega$$
 indicating that $V_1|_{\Omega}$ is well-defined and of class of $C^1$ in $H^1_0(\Omega)$
 endowed with its usual norm.

 Moreover,
 we define the ground state $c_0$
associated with \eqref{mainequation3}
by
$$m_\Omega\triangleq\inf_{u\in \mathcal{N}^{R,\bar{\delta}}_\Omega}J_\Omega^{R,\bar{\delta}}(u),
~\text{where}~
\mathcal{N}^{R,\bar{\delta}}_\Omega  =\{u\in H_0^1(\Omega)\backslash\{0\}:u\neq0,~ (J_\Omega^{R,\bar{\delta}} )^\prime(u)[u]=0  \}.
$$

Now, we are ready to prove Theorem \ref{maintheorem3} as follows.

 \begin{proof}[\textbf{\emph{Proof of Theorem \ref{maintheorem3}}}]
 Let $u_\lambda\in X_\lambda$ be a
ground state solution for Eq. \eqref{mainequation1},
choosing a subsequence $\lambda_n\to+\infty$
as $n\to\infty$, we denoted $\{u_{\lambda_n}\}$ by the subsequence of $\{u_\lambda\}$.
In view of the proof of Theorem \ref{maintheorem1}, we know that
\begin{equation}\label{Proof31}
 \sup_{n\in \mathbb{N}}\|u_{\lambda_n}\|_{H^1(\R^2)} ^2< \frac{2\pi K}{q}~\text{and}~
 \sup_{n\in \mathbb{N}}|u_{\lambda_n}|_\infty \leq C_0
\end{equation}
and
\begin{equation}\label{Proof32}
 \limsup_{n\to\infty}\|u_{\lambda_n}\|_{H^1(\R^2)}^2 <\frac{\pi}{\gamma+\alpha e^{\frac1e}}~\text{and}~
 \sup_{n\in \mathbb{N}}|u_{\lambda_n}|_\infty \leq C_0^\prime
\end{equation}
for $\bar{\delta}=\delta\in(0,2)$ and $\bar{\delta}=2$, respectively.
Moreover, due to the proof of Lemma \ref{bounded},
we deduce that $\|u_{\lambda_n}\|_{X_{\lambda_n}}$
is uniformly bounded in $n\in \mathbb{N}$ since we have showed the facts
\eqref{Ia2} and \eqref{IIa1} are true for all $\lambda\geq\lambda_0$
if $\bar{\delta}=\delta\in(0,2)$, or $\lambda\geq\lambda_0^\prime$ if $\bar{\delta}=2$.
Going to a subsequence if necessary,
there is a $u\in X$
such that $u_{\lambda_n}\rightharpoonup u$ in $X$,
$u_{\lambda_n}\to u$ in $L^s(\R^2)$ for each $2\leq s<\infty$ by Lemma \ref{zz}-(i)
and $u_{\lambda_n}\to u$ a.e. in $\R^2$ as $n\to\infty$.

We claim that $u\equiv0$ in $\Omega^c$.
Otherwise, there is a compact subset $\Theta_u\subset\Omega^c$ with
$\text{dist}(\Theta_u,\partial\Omega^c)>0$
such that $u\neq0$ on $\Theta_u$ and by Fatou's lemma
\begin{equation}\label{concentrating2}
	\liminf_{n\to\infty}\int_{\R^2}u_{\lambda_n}^2dx\geq \int_{\Theta_u}u^2dx>0.
\end{equation}
Moreover, there exists $\varepsilon_0>0$ such that $V(x)\geq \varepsilon_0$ for any
$x\in \Theta_u$ by the assumptions $(V_1)$ and $(V_2)$. Combining \eqref{ff} and \eqref{concentrating2}, one has
\begin{align*}
	c_{\lambda_n}^{R,\bar{\delta}}&= \liminf_{n\to\infty}J^{R,\bar{\delta}}_{\lambda_n}(u_{\lambda_n})
	=\liminf_{n\to\infty}\big[ J_{\lambda_n}^{R,\bar{\delta}}(u_{\lambda_n})-\frac14 (J_{\lambda_n}^{R,\bar{\delta}})^\prime(u_{\lambda_n})[u_{\lambda_n}]\big]  \\
	&\geq   \frac14\liminf_{n\to\infty}\int_{\R^2}\lambda_nV(x)|u_{\lambda_n}|^2dx
	\geq \frac{\varepsilon_0}4\bigg(\int_{\Theta_u}
	u^2dx\bigg)\liminf_{n\to\infty}\lambda_n=+\infty
\end{align*}
violating \eqref{Ia2} and \eqref{IIa1}.
Consequently, $u\in H_0^1(\Omega)$ by the fact that $\partial \Omega$ is smooth.

Recalling $(J_{\lambda_n}^{R,\bar{\delta}})^\prime(u_{\lambda_n})=0$ for each fixed $n\in \mathbb{N}$,
we now claim that $(J_{\Omega}^{R,\bar{\delta}})^\prime(u)=0$. In fact,
for every $\psi\in C_0^\infty(\Omega)$,
 it is very similar to \eqref{2Proof1} and \eqref{2Proof2} that
\begin{equation}\label{propo5e}
\begin{gathered}
 V_{i}'(u_{\lambda_n})[\psi]-V_{i}'(u)[\psi]
   =  4B_i(u_{\lambda_n}^2, u_{\lambda_n}\psi)-4B_i(u^2, u\psi)\hfill\\
 \ \ \ \    =4B_i(u_{\lambda_n}^2, (u_{\lambda_n}-u)\psi)
 +4B_i(u\psi,u_{\lambda_n}^2-u^2) =o_n(1).
 \hfill\\
 \end{gathered}
\end{equation}
By means of the Vitali's Dominated Convergence theorem,
one can easily verify that
\begin{equation}\label{propo5f}
\int_{\R^2} f^{R,\bar{\delta}}(u_{\lambda_n})\psi dx=\int_{\R^2} f^{R,\bar{\delta}}(u )\psi dx+o_n(1).
\end{equation}
With the above two formulas in hand, by using $(V_2)$, we derive
\begin{align*}
  0 &= (J_{\lambda_n}^{R,\bar{\delta}})^\prime(u_{\lambda_n})[\psi] \\
    &  =\int_{\R^2} \nabla u_{\lambda_n}\nabla \psi dx+ V_1^\prime(u_{\lambda_n})[\psi]-V_2^\prime(u_{\lambda_n})[\psi]
- \int_{\R^2} f^{R,\bar{\delta}}(u_{\lambda_n})\psi dx\\
&=\int_{\Omega} \nabla u \nabla \psi dx+ V_1^\prime|_\Omega(u )[\psi]-V_2|_\Omega^\prime(u )[\psi]
- \int_{\R^2} f^{R,\bar{\delta}}(u )\psi dx+o_n(1)\\
&=(J_{\Omega}^{R,\bar{\delta}})^\prime(u )[\psi]+o_n(1),
\end{align*}
which is the claim.
Next, we begin showing that $u\neq0$. Arguing it indirectly and supposing that
$u\equiv0$.
Since $\{u_{\lambda_n}\}$ is
uniformly bounded in $L^2(\R^2)$, and so, recalling $(V_3)$, we can proceed as the calculations in the proof
of Lemma \ref{bounded} to find
a sufficiently large constant $\varrho>0$ such that
$$
\int_{B_\varrho^c(0)\cap \Xi}|u_{\lambda_n}|^2dx\leq \frac{\beta_0}4,~\text{for \emph{n} sufficiently large}.
$$
Adopting $V(x)\geq b$ on $\Xi^c$ by $(V_3)$ and $\|u_{\lambda_n}\|_{X_{\lambda_n}}$
is uniformly bounded, there holds
$$
\int_{B_\varrho^c(0)\cap \Xi^c}|u_{\lambda_n}|^2dx\leq \frac{1}{\lambda_nb} \int_{B_\varrho^c(0)\cap \Xi^c}\lambda_nV(x)|u_{\lambda_n}|^2dx
\leq \frac{\beta_0}4,~\text{for \emph{n} sufficiently large}.
$$
Let us recall that $u_{\lambda_n}\to 0$ along a subsequence in $L^2(B_\varrho(0))$, then
$$
\int_{B_\varrho (0) }|u_{\lambda_n}|^2dx
\leq \frac{\beta_0}4,~\text{for \emph{n} sufficiently large}.
$$
Combining the above three formulas, there would be a contradiction to Lemma \ref{Non-Vanishing}.
So, $u\neq0$.

Finally, the remaining part is to verify $J^{R,\bar{\delta}}_{\Omega}(u)=c^{R,\bar{\delta}}_\Omega$
because each of \eqref{Proof31} and \eqref{Proof32} indicates that
$F^{R,\bar{\delta}}=F$ which yields that $J_\Omega^{R,\bar{\delta}}=J_\Omega$.
Obviously, $\mathcal{N}^{R,\bar{\delta}}_\Omega\subset \mathcal{N}^{R,\bar{\delta}}_{\lambda_n}$
for all fixed $n\in \mathbb{N}$
and so we have that
\begin{align*}
 c^{R,\bar{\delta}}_\Omega&
\ge\liminf_{n\to\infty}
 \big(J_{\lambda_n}^{R,\bar{\delta}}(u_{\lambda_n})-\frac{1}{4} (J_{\lambda_n}^{R,\bar{\delta}})^\prime(u_{\lambda_n})[u_{\lambda_n}]\big)
    \geq J_{\Omega}^{R,\bar{\delta}}(u )-\frac{1}{4} (J_{\Omega}^{R,\bar{\delta}})^\prime(u )[u ]=J_{\Omega}^{R,\bar{\delta}}(u)\geq c_\Omega^{R,\bar{\delta}}
\end{align*}
indicating that $u_{\lambda_n}\to u$ in $X$ and $J^{R,\bar{\delta}}_{\Omega}(u)=c^{R,\bar{\delta}}_\Omega$,
where we have exploited the Fatou's lemma together with $(J_{\Omega}^{R,\bar{\delta}})^\prime(u)=0$ and $u\neq0$.
 The proof is finished.
  \end{proof}

\section{Some further remarks}\label{Sec5}

In this section, we tend to present some further results on Eq. \eqref{mainequation1}
with steep potential well and supercritical exponential growth.
Inspired by \cite[Theorem 1.3]{ZLZ},
one could prove that the assumption $(h_2)$ is unnecessary
to manipulate Theorems \ref{maintheorem1} and \ref{maintheorem3}.
Speaking simply, there are a wider class of nonlinearities
which are suitable for the main results in this article. Hence, we shall mainly discuss how to
relax the restriction associated with $(h_2)$. For this purpose, we suppose that
\begin{itemize}
  \item[$(h_{21})$] There is a constant $\theta>3$ such that $h(t)t\geq \theta H(t)$ for all $t\in \R^+$;
  \item[$(h_{22})$] The function $h(t)/t^2$ is nondecreasing on $t\in \R^+$.
\end{itemize}
Concerning the potential $V$, we need to put forward some additional conditions below
 \begin{itemize}
  \item[$(V_{4}^1)$] $V\in C^1(\R^2,\R)$ and $2 V(x)+(\nabla V,x)\geq0$ for all $x\in\R^2$;
  \item[$(V_{4}^2)$]  $V\in C^1(\R^2,\R)$ and $t \mapsto t^2[2V(tx)-(\nabla V (tx),tx)]$ is nondecreasing on $t\in \R^+$ for all $x\in\R^2$
as well as $V(x)-(\nabla V (x),x)\geq0$ for all $x\in\R^2$.
\end{itemize}

The main results in this section could be stated as follows.

\begin{theorem}\label{maintheorem4}
Let $V$ satisfy $(V_1)-(V_3)$ and $(V_{4}^1)$.
Suppose that the nonlinearity $f$ defined in \eqref{form}
requires $(h_1),(h_3)$ and $(h_{21})$, then for all
$\tau>2$, there are $\alpha^*_1=\alpha^*_1(\tau)>0$ and $\lambda_0^1>0$ such that Eq. \eqref{mainequation1}
has a nonnegative nontrivial solution in $X_\lambda$ for all $\alpha \in (0, \alpha^*_1)$ and $\lambda>\lambda_0^1$.
Moreover, if we suppose that
\begin{itemize}
\item[$(h_4)$] there are consatnts $\xi>0$ and $p>4$ such that $H(t)=\int_0^th(s)ds\geq \xi t^{p}$ for all $t\in[0,1]$,
\end{itemize}
then for every
$\alpha>0$, there exist $\tau_*^1=\tau_*^1(\alpha)>2$, $\bar{\lambda}_0^\prime>0$ and ${\xi}_0^1>0$
such that Eq. \eqref{mainequation1} possesses a nonnegative nontrivial solution in $X_\lambda$ for every $\tau \in [2,\tau_*^1)$,
$\lambda>\bar{\lambda}_0^\prime$ and ${\xi}>{\xi}_0^1$.
\end{theorem}

\begin{theorem}\label{maintheorem5}
Let $V$ satisfy $(V_1)-(V_3)$ and $(V_{4}^1)-(V_{4}^2)$.
Suppose that the nonlinearity $f$ defined in \eqref{form}
requires $(h_1),(h_3)$ and $(h_{21})-(h_{22})$, then for all
$\tau>2$, there are $\alpha^*_2=\alpha^*_2(\tau)>0$ and $\lambda_0^2>0$ such that Eq. \eqref{mainequation1}
has a nonnegative ground state solution in $X_\lambda$ for all $\alpha \in (0, \alpha^*_2)$ and $\lambda>\lambda_0^2$.
Moreover, if we suppose that
\begin{itemize}
\item[$(h_4)$] there are constants $\xi>0$ and $p>4$ such that $H(t)=\int_0^th(s)ds\geq \xi t^{p}$ for all $t\in[0,1]$,
\end{itemize}
then for every
$\alpha>0$, there exist $\tau_*^2=\tau_*^2(\alpha)>2$, $\tilde{\lambda}_0^\prime>0$ and ${\xi}_0^2>0$
such that Eq. \eqref{mainequation1} possesses a nonnegative ground state solution in $X_\lambda$ for every $\tau \in [2,\tau_*^2)$,
$\lambda>\tilde{\lambda}_0^\prime$ and ${\xi}>{\xi}_0^2$.
\end{theorem}

\begin{remark}
 It is simple to observe that $(h_2)$ is definitely stronger than $(h_{21})$ and $(h_{22})$.
 Chen \emph{et al.} \cite{CSTW} obtained the existence of nontrivial solutions
 for Eq. \eqref{mainequation2} with $\la\equiv0$ and $\theta=3$ in $(h_{2}^1)$,
 but the novelty is that we consider $\lambda>0$
 and $f$ can be admitted a supercritical exponential growth.
\end{remark}

\begin{remark}
There are a lot of functions $V$ satisfying $(V_1)-(V_3)$ and $(V_4^1)-(V_4^2)$. For example,
in the sense of ignoring a zero measure set, we define $V(x)=V(|x|)$ for a.e. $x\in\R^2$ given by
\[
V(x)=\left\{
       \begin{array}{ll}
         0, & 0\leq |x|\leq1. \\
         |x|, & |x|\geq1.
       \end{array}
     \right.
\]
\end{remark}

Due to the discussions in the Introduction, to accomplish the proofs of Theorems \ref{maintheorem4}
and \ref{maintheorem5}, we have to establish the following two results.

\begin{theorem}\label{maintheorem6}
Let $V$ satisfy $(V_1)-(V_3)$ and $(V_{4}^1)$.
Suppose that the nonlinearity $f$ defined in \eqref{form}
requires $(h_1),(h_3)$ and $(h_{21})$, then for each fixed $R>0$, there is a
$\bar{\lambda}_0(R)>0$ dependent of $R$ such that Eq. \eqref{mainequation2}
with $\bar{\delta}=\delta$
has a nonnegative nontrivial solution in $X_\lambda$ for all $\lambda>\bar{\lambda}_0(R)$.
Moreover, if in addition we suppose that $(h_4)$,
then there exist $\bar{\lambda}_0^\prime(R)>0$ and ${\bar{\xi}}_0(R)>0$
such that Eq. \eqref{mainequation2} with $\bar{\delta}=2$ possesses a nonnegative nontrivial solution in $X_\lambda$ for all
$\lambda>\bar{\lambda}_0^\prime(R)$ and ${\xi}>{\bar{\xi}}_0(R)$.
\end{theorem}

\begin{theorem}\label{maintheorem7}
Let $V$ satisfy $(V_1)-(V_3)$ and $(V_{4}^1)-(V_{4}^2)$.
Suppose that the nonlinearity $f$ defined in \eqref{form}
requires $(h_1),(h_3)$ and $(h_{21})-(h_{22})$, then for each fixed $R>0$, there is a
$\tilde{\lambda}_0(R)>0$ dependent of $R$ such that Eq. \eqref{mainequation2}
with $\bar{\delta}=\delta$
has a nonnegative ground state solution in $X_\lambda$ for all $\lambda>\tilde{\lambda}_0(R)$.
Moreover, if in addition we suppose that $(h_4)$,
then there exist $\tilde{\lambda}^\prime(R)>0$ and $\tilde{\xi}_0(R)>0$
such that Eq. \eqref{mainequation2} with $\bar{\delta}=2$ admits a nonnegative ground state solution in $X_\lambda$ for all
$\lambda>\tilde{\lambda}_0^\prime(R)$ and ${\xi}> \tilde{\xi}_0(R)$.
\end{theorem}

\begin{remark}
Without considering the computations in Theorem \ref{maintheorem6} or \ref{maintheorem7},
combined the arguments in Section \ref{Sec3} and \cite[Theorem 1.1]{CT1},
maybe one could get the desired results. Nevertheless,
the most challenging difficulty is how to investigate the $L^\infty$-estimate of the solution
obtained by Theorem \ref{maintheorem6} or \ref{maintheorem7}. To be more eloquent, because of the lack of \eqref{Vanishing2},
we must come up with some new analytic tricks to receive the analogous
estimate in \eqref{Ia1} or \eqref{IIa4}. Let us point out here that although
it can be obtained easily with the help of the method in \cite{Shen} which should make the nonlinearity $f$ have to be imposed on
some more stronger restrictions,
 this is not what we expected in this section.

Finally, we shall invite the reader to note that
the so-called ground state in Theorems \ref{maintheorem5} and \ref{maintheorem7} is of class Poho\u{z}aaev-Nehari type,
instead of the Nehari type in Theorems \ref{maintheorem1} and \ref{maintheorem2}.
\end{remark}

Following a very similar calculations in Theorem \ref{maintheorem3},
we could conclude the asymptotical behavior of $u_\lambda$ which is the solution established by
Theorem \ref{maintheorem4} (or Theorem \ref{maintheorem5}). Let us state it without
 the detailed proof as follows.

\begin{theorem}\label{maintheorem8}
Under the assumptions in Theorem \ref{maintheorem4} (or Theorem \ref{maintheorem5}), passing to a subsequence,
$u_\lambda\to   u_0$ in $X$ as $\lambda\to+\infty$, where $u_0$
is a nontrivial (or ground state) solution for the Schr\"{o}dinger-Poisson equation
 \begin{equation}\label{mainequation3}
\left\{%
\begin{array}{ll}
\displaystyle     -\Delta u+\bigg(\int_{\Omega}\log|x-y|u^2(y)dy\bigg)   u=f(u) , & x\in\Omega, \\
     u=0, &  x\in\partial\Omega.\\
\end{array}%
\right.
\end{equation}
\end{theorem}

Now, we start with some brief proofs of the main results in Theorems \ref{maintheorem4}, \ref{maintheorem5},
\ref{maintheorem6} and \ref{maintheorem7} in this section.

\subsection{Proofs of Theorems \ref{maintheorem6}
and \ref{maintheorem7}}
The proofs are
  divided into two parts below.

\subsubsection{\emph{Proof of Theorem \ref{maintheorem6}}}\

We firstly recall the following result introduced by Jeanjean \cite{Jeanjean1,Jeanjean2} which is crucial for finding a
bounded $(PS)$ sequence without the Ambrosetti-Rabinowitz condition, namely $\theta\geq4$ in $(h_2^1)$.

\begin{proposition}\label{Jeanjean}
Let $(Z,\|\cdot\|_Z)$ be a Banach space and let $\Lambda\subset \R^+$ be an interval; consider a
family $\{\Phi_\mu\}_{\mu\in\Lambda}$ on $Z$ of $\mathcal{C}^1$-functional having the form
$$
\Phi_\mu(u)=A(u)-\mu B(u),~\forall \mu\in \Lambda,
$$
with $B(u)\geq 0$ for all $u\in Z$ and either $A(u)\to+\infty$ or $B(u)\to+\infty$ as $\|u\|_Z\to+\infty$. Assume that there are two points
$v_1, v_2\in Z$ such that
\[
c_\mu=\inf_{\gamma\in\Gamma}\max_{t\in[0,1]}\Phi_\mu(\gamma(t))>\max\{\Phi_\mu(v_1),\Phi_\mu(v_2)\},
\]
where $\Gamma=\{\gamma\in \mathcal{C}([0,1],X):\gamma(0)=0,\gamma(1)=v_2\}$.
Then, for a.e. $\mu\in \Lambda$, there is a bounded $(PS)_{c_\mu}$
sequence for $\Phi_\mu$ in $Z$, namely a sequence $\{u_n(\lambda)\}\subset X$ satisfying
\begin{itemize}
  \item[\emph{(i)}] $\{u_n(\lambda)\}$ is bounded in $Z$;
  \item[\emph{(ii)}] $\Phi_\mu(u_n(\lambda))\to c_\mu$ as $n\to\infty$;
  \item[\emph{(iii)}] $\Phi^\prime_\mu(u_n(\lambda))\to0$ in $Z^{-1}$.
\end{itemize}
Moreover, $c_\mu$ is non-increasing on $\mu\in \Lambda$.
\end{proposition}

To apply Proposition \ref{Jeanjean} successfully, inspired by \cite{CT1},
we have to modify the work space $(X_\lambda,\|\cdot\|_{X_\lambda})$ mildly. Speaking clearly,
according to Remark \ref{remark2.2}, for all $\lambda\geq1$, we redefine the space $X_\lambda$ by
\[
X_\lambda \triangleq\bigg\{u\in E_\lambda:\int_{\R^2}\log(2+|x|)|u|^2dx<+\infty\bigg\}
\]
which is a Hilbert space
equipped with the inner product and norm
\[
(u,v)_{X_\lambda}=\int_{\R^2}\big[\nabla u\nabla v+(\lambda V(x)+\log(2+|x|))uv \big]dx~\text{and}~
\|u\|_{X_\lambda} =\sqrt{(u,u)_{X_\lambda}}, ~\forall u,v\in {X_\lambda}.
\]
So, $\|\cdot\|_{X_\lambda}=\sqrt{\|\cdot\|^2_{E_\lambda}+\|\cdot\|_*^2}$,
where $\|u\|_*= (\int_{\R^2} \log(2+|x|)|u|^2dx )^{\frac12}$ for all $u\in X$, where
\[
X=\bigg\{u\in H^1(\R^2):
\int_{\R^2}\log(2+|x|)u^2 dx<+\infty\bigg\}.
\]
 With this space $X$, we rewrite $V_1$ and $V_2$ on $X$ in this subsection below
\[
V_1(u)\triangleq\int_{\R^2}\int_{\R^2}\log(2+|x-y|)u^2(x)u^2(y)dxdy,~ \forall u\in X,
\]
and
\[
V_2(u)\triangleq\int_{\R^2}\int_{\R^2}\log\bigg(1+\frac{2}{|x-y|}\bigg)u^2(x)u^2(y)dxdy,~  \forall u\in X.
\]
As one could observe that there is no essential difference between the above definitions and those in Sections
\ref{Introduction}-\ref{Preliminaries}. Thereby,
we keep the same notations in this subsection just for simplicity when there is no misunderstanding.

Setting $\Phi_\mu(u)=J_{\lambda,\mu}^{R,\bar{\delta}}(u)$ on the work space $(Z,\|\cdot\|_Z)=(X_\lambda,\|\cdot\|_{X_\lambda})$
with $\lambda\geq1$, where
\[
J_{\lambda,\mu}^{R,\bar{\delta}}(u)=\frac{1}{2}\int_{\R^2}[|\nabla u|^2+\lambda V(x)|u|^2]dx+ \frac12\|u\|_*^2 + \frac14V_0(u)
- \mu\bigg(\int_{\R^2} F^{R,\bar{\delta}}(u)dx+ \frac12\|u\|_*^2  \bigg)
\]
for all $\mu\in[\frac12,1]$. Let us rewrite $J_{\lambda,\mu}^{R,\bar{\delta}}(u)=A(u)-\mu B(u)$ on $X_\lambda$ with
\[
A(u)\triangleq\frac{1}{2}\int_{\R^2}[|\nabla u|^2+\lambda V(x)|u|^2]dx+ \frac12\|u\|_*^2 + \frac14V_0(u),~
B(u)\triangleq\int_{\R^2} F^{R,\bar{\delta}}(u)dx+ \frac12\|u\|_*^2.
\]

\begin{lemma}\label{propo}
Let $V$ satisfy $(V_1)-(V_3)$. Suppose that $f$ defined by \eqref{form}
satisfies $(h_1)$ and $(h_3)$, then for all fixed $R>0$ and $\lambda\geq1$,
then $B(u)\geq0$ for any $u\in X_\lambda$. Moreover, either $A(u)\to+\infty$
or $B(u)\to+\infty$ as $\|u\|_{X_\lambda}\to+\infty$.
\end{lemma}

\begin{proof}
By $(h_3)$, one observes that $B(u)\geq0$ for each $u\in X_\lambda$. In consideration of the completeness,
we borrow the ideas in \cite[Lemma 3.2]{CT1}
to conclude the proof.
Arguing it indirectly, we could suppose that, up to a subsequence if necessary, there is a sequence
$\{u_n\}\subset X_\lambda$ such that
\begin{equation}\label{propo1}
\|u_n\|_{X_\lambda}\to+\infty,~A(u_n)\leq C~\text{and}~B(u_n)\leq C.
\end{equation}
So, recalling the definition of $\|\cdot\|_*$ in this subsection and \eqref{propo1},
\[
|u_n|_2^2\leq \frac{1}{\log 2}\|u_n\|_*^2\leq \frac{1}{\log 2}B(u_n)\leq C
\]
which together with \eqref{v2} and the Gagliardo-Nirenberg inequality indicates that
\[
V_2(u_n)\leq K_0|u_n|_{\frac83}^4\leq C|u_n|_2^3|\nabla u_n|_2\leq C\|u_n\|_{X_\lambda}.
\]
As a consequence of the above formula and \eqref{propo1}, there holds
\[
C\geq A(u_n)\geq \frac12\|u_n\|_{X_\lambda}^2-\frac14V_2(u_n)\geq \frac12\|u_n\|_{X_\lambda}^2-\frac C4\|u_n\|_{X_\lambda}
\]
which contradicts with $\|u_n\|_{X_\lambda}\to+\infty$. The proof is complete.
\end{proof}

\begin{lemma}\label{propo2}
Let $V$ satisfy $(V_1)-(V_3)$. Suppose that $f$ defined by \eqref{form}
satisfies $(h_1)$ and $(h_3)$, then for all fixed $R>0$ and $\lambda\geq1$,
then we have that
\begin{itemize}
  \item[\emph{(i)}] There exists a $v_0\in  X_\lambda\backslash \{0\}$ independent of $\mu$ and $\lambda$ such that
$J_{\lambda,\mu}^{R,\bar{\delta}}(v_0)\leq0$ for all $\mu\in[\frac12,1]$;
  \item[\emph{(ii)}] Denoting $\Gamma_{\lambda,\mu}^{R,\bar{\delta}}=\{\gamma\in \mathcal{C}([0,1],X_\lambda):\gamma(0)=0,\gamma(1)=v_0\}$, then
\[
c_{\lambda,\mu}^{R,\bar{\delta}}\triangleq\inf_{\gamma\in\Gamma_{\lambda,\mu}^{R,\bar{\delta}}}\max_{t\in[0,1]}J_{\lambda,\mu}^{R,\bar{\delta}}
(\gamma(t))\geq A_1>\max\{J_{\lambda,\mu}^{R,\bar{\delta}}(0),J_{\lambda,\mu}^{R,\bar{\delta}}(v_0)\},~\forall \mu\in[\frac12,1];
\]
\item[\emph{(iii)}]  There exists a constant $M^{R,\delta} > 0$ independent of $\mu$ and $\lambda$ such that
$c_{\lambda,\mu}^{R,{\delta}}\leq M^{R,\delta}$
\end{itemize}
\end{lemma}

\begin{proof}
(i) Without loss of generality, we suppose that $0\in\Omega$ and then
there exists a $\varepsilon_0>0$ such that $B_{\varepsilon_0}(0)\subset \Omega$ by $(V_2)$.
Let $\psi\in C_0^\infty(B_{\varepsilon_0}(0))$ and set $\psi_t\triangleq t^2\psi(t\cdot)$
for all $t>1$, hence $\supp \psi_t\in B_{\varepsilon_0}(0)$. By some direct computations, we have
\begin{equation}\label{Propo2a0}
\left\{
  \begin{array}{ll}
\displaystyle   \int_{\R^3}|\nabla\psi_t|^2dx=t^4\int_{B_{\varepsilon_0}(0)}|\nabla\psi|^2dx,~\int_{\R^2} F^{R,\bar{\delta}}(\psi_t)dx
=t^{-2}\int_{B_{\varepsilon_0}(0)} F^{R,\bar{\delta}}(t\psi )dx,\\
\displaystyle     \int_{\R^2} \log(2+|x|)|\psi_t|^2dx=t^2\int_{B_{\varepsilon_0}(0)} \log(2+t^{-1}|x|)|\psi|^2dx
\leq \log(2+\varepsilon_0 )t^2\int_{B_{\varepsilon_0}(0)}|\psi|^2dx,\\
 \displaystyle   V_0(\psi_t)=t^4V_0(\psi)
 - t^4 \log t |\psi|_2^4\leq t^4 \bigg[\log(2+\varepsilon_0 ) \int_{B_{\varepsilon_0}(0)}|\psi|^2dx+\log t \bigg(\int_{B_{\varepsilon_0}(0)}|\psi|^2dx\bigg)^2\bigg].
  \end{array}
\right.
\end{equation}
Due to $(h_1)$ and $(h_3)$ together with \eqref{Propo2a0}, there holds $F^{R,\bar{\delta}}(t)\geq C_1 t^{\theta}-C_2t^2$ for all $t>0$
and thus
\begin{equation}\label{Propo2a}
\frac{J_{\lambda,\frac12}^{R,\bar{\delta}}(\psi_t)}{t^4\log t}\leq\frac{1}{4}\bigg(\int_{B_{\varepsilon_0}(0)}|\psi|^2dx\bigg)^2
- \frac{t^{2(\theta-3)}}{2\log t}\int_{\R^2} \frac{F^{R,\bar{\delta}}(t^2\psi)}{(t^2\psi)^\theta}\psi ^\theta dx+o_t(1)\to-\infty
\end{equation}
 as $t\to+\infty$ because $\theta>3$. Choosing $v_0= \psi_t$ with $t$ sufficiently large,
 we have $J_{\lambda,\mu}^{R,\bar{\delta}}(v_0)\leq J_{\lambda,\frac12}^{R,\bar{\delta}}(v_0)<0$ for all $\mu\in[\frac12,1]$.

(ii) We can repeat the calculations in Lemma
\ref{geometry}-(i) to find such a constant $A_1>0$ independent of $\mu\in[\frac12,1]$ and $\lambda\geq1$.

(iii) Since $c_{\lambda,\mu}^{R, {\delta}}\leq \max_{t>0}
J_{\lambda,\mu}^{R,\bar{\delta}}(\psi_t)\leq  J_{\lambda,\frac12}^{R,\bar{\delta}}(\psi_t)$
for all $\mu\in[\frac12,1]$, the conclusion follows
from \eqref{Propo2a} immediately. Thus, we finish the proof of this lemma.
\end{proof}

\begin{lemma}\label{2Propo2}
Under the assumptions of Lemma \ref{propo2}, if in addition we suppose that $(h_4)$,
then there exists a $\underline{\xi}_0=\underline{\xi}_0(R)>0$
 such that for all $\xi>\bar{\xi}_0$,
 \[
 c_{\lambda,\mu}^{R,{2}}<\frac{\pi(\theta-2)}{2\theta C_{\Xi,b}(1+\gamma+\alpha R^{\tau-2})},~\forall \lambda\geq1,~\mu\in[\frac12,1],
 \]
 where $\gamma>2$, $\theta>3$ and $C_{\Xi,b}>0$ come from $(h_3)$, $(h_{21})$
 and Lemma \ref{imbedding}, respectively.
\end{lemma}

\begin{proof}
We postpone the proof and refer the reader to Lemma \ref{22Propo2} below for the details.
\end{proof}

It is essentially same as the proof of \cite[Lemma 2.4]{Du}, we can derive the following lemma.

\begin{lemma}\label{Pohozaev}
Let $V$ satisfy $(V_1)-(V_3)$ and $(V_4^1)$. Assume that
 the nonlinearity $f$ defined in \eqref{form}
requires $(h_1)$ and $(h_3)$. For each fixed $R>0$ and $\lambda\geq1$,
if $\bar{u}\in X_\lambda$ is a critical point of $J_{\lambda,\mu}^{R,\bar{\delta}}$,
then it satires the Poho\u{z}aev identity $P_{\lambda,\mu}^{R,\bar{\delta}}(\bar{u})\equiv0$,
where $P_{\lambda,\mu}^{R,\bar{\delta}}:X_\lambda\to\R$ is given by
\begin{align*}
P_{\lambda,\mu}^{R,\bar{\delta}}(u)& =\frac{1}{2}\int_{\R^2}\lambda[ 2V(x)+(\nabla V,x)]|u|^2dx+\frac14|u|_2^4+ V_0(u) \\
  & \ \ \  +\frac{1-\mu}{2}\bigg(2\|u\|_*^2 +\int_{\R^2}\frac{|x|}{2+|x|}u^2dx\bigg) - 2\mu\int_{\R^2} F^{R,\bar{\delta}}(u)dx.
\end{align*}
\end{lemma}

\begin{lemma}\label{propo5}
Let $V$ satisfy $(V_1)-(V_3)$ and $(V_4^1)$. Assume that
 the nonlinearity $f$ defined in \eqref{form}
requires $(h_1),(h_3)$ and $(h_{21})$, then for each fixed $R>0$ and $\lambda\geq1$,
there is a $u_\mu\in X_\lambda\backslash\{0\}$ such that
\begin{equation}\label{propo5a}
 (J_{\lambda,\mu}^{R,\bar{\delta}})^\prime(u_\mu)=0~\text{and}~
 J_{\lambda,\mu}^{R,\bar{\delta}}(u_\mu)\in (0, c_{\lambda,\mu}^{R,\bar{\delta}}],~\forall a.e.~ \mu\in[\frac12,1],
\end{equation}
where we must suppose additionally that $(h_4)$ with $\xi>\underline{\xi}_0$ appearing in Lemma \ref{2Propo2}
whence $\bar{\delta}=2$
\end{lemma}

\begin{proof}
Combining Proposition \ref{Jeanjean} and Lemmas \ref{propo} and \ref{propo2},
for a.e. $\mu\in[\frac12,1]$,
there is a sequence $\{u_n(\mu)\}\subset X_\lambda$ (we denote it by $\{u_n\}$ just for short) such that
\begin{equation}\label{propo5b}
   \|u_n\|_{X_\lambda}\leq C,~J_{\lambda,\mu}^{R,\bar{\delta}}(u_n)\to c_{\lambda,\mu}^{R,\bar{\delta}}>0~\text{and}~
(J_{\lambda,\mu}^{R,\bar{\delta}})^\prime(u_n)\to0~\text{in}~X_\lambda^{-1}.
\end{equation}
Since $\|u_n\|_{X_\lambda}\leq C$, passing to a subsequence if necessary, there is a $u_\mu\in X_\lambda$
such that $u_n\rightharpoonup u_\mu$ in $X_\lambda$, $u_n\to u_\mu$ in $L^s(\R^2)$
for every $s\in[2,+\infty)$ and $u_n\to u_\mu$ a.e. in $\R^2$.
We then claim that $u_\mu\neq0$. Supposing it by a contradiction,
using Lemma \ref{zz}-(ii) and (iv), we have
\begin{equation}\label{propo5c}
|V_0(u_n)|\leq V_1(u_n)+V_2(u_n)\leq 2|u_n|_2^2\|u_n\|_*^2+K_0|u_n|_{\frac{8}{3}}^{4}=o_n(1),
\end{equation}
where we have used the fact $\|u_n\|_*^2\leq C^2$. By means of \eqref{propo5b} and \eqref{propo5c} jointly with $(h_{21})$, there holds
\begin{align*}
  c_{\lambda,\mu}^{R,\bar{\delta}} &= J_{\lambda,\mu}^{R,\bar{\delta}}(u_n)-\frac1\theta
   (J_{\lambda,\mu}^{R,\bar{\delta}})^\prime(u_n)[u_n] +o_n(1)\\
   & =\frac{2\theta}{\theta-2}\|u_n\|_{E_\lambda}^2+\frac{2\theta(1-\mu)}{\theta-2}\|u_n\|_{*}^2+\frac{\mu}{\theta}
\int_{\R^2}[f^{R,\bar{\delta}}(u_n)u_n-\theta F^{R,\bar{\delta}}(u_n) ]dx+o_n(1)\\
&\geq \frac{2\theta}{\theta-2}\|u_n\|_{E_\lambda}^2+o_n(1)
\end{align*}
which together with Lemma \ref{propo2}-(iii) and Lemma \ref{2Propo2}
reveals that all the assumptions in Lemma \ref{compact} holds true.
Consequently, we deduce that
\begin{equation}\label{propo5d}
 \lim_{n\to\infty}\int_{\R^2}f^{R,\bar{\delta}}(u_n)u_ndx= 0~\text{and}
 \lim_{n\to\infty}\int_{\R^2}F^{R,\bar{\delta}}(u_n)dx= 0
\end{equation}
With \eqref{propo5c} and \eqref{propo5d} in hand,
we could derive $\|u_n\|_{X_\lambda}\to0$ by $(J_{\lambda,\mu}^{R,\bar{\delta}})^\prime(u_n)[u_n]\to0$.
As a consequence, it holds that $0=\lim\limits_{n\to\infty}J_{\lambda,\mu}^{R,\bar{\delta}}(u_n)=
c_{\lambda,\mu}^{R,{2}}$ which is absurd since Lemma \ref{propo2}-(ii). So, $u_\mu\neq0$ is true.

Next, we shall show that $(J_{\lambda,\mu}^{R,\bar{\delta}})^\prime(u_\mu)=0$.
To see it, for all $\phi\in C_0^\infty(\R^2)$, proceeding as \eqref{propo5e} and \eqref{propo5f}, there holds
$(J_{\lambda,\mu}^{R,\bar{\delta}})^\prime(u_\mu)[\psi]=\lim\limits_{n\to\infty}
(J_{\lambda,\mu}^{R,\bar{\delta}})^\prime(u_n)[\psi]=0$ which indicates the desired result.
Using Lemma \ref{Pohozaev}, we obtain $P_{\lambda,\mu}^{R,\bar{\delta}}(u_\mu)=0$ and so
\begin{align*}
J_{\lambda,\mu}^{R,\bar{\delta}}(u_\mu)&= J_{\lambda,\mu}^{R,\bar{\delta}}(u_\mu)-\frac14\big[
2(J_{\lambda,\mu}^{R,\bar{\delta}})^\prime(u_\mu)[u_\mu] -P_{\lambda,\mu}^{R,\bar{\delta}}(u_\mu)\big]    \\
   & =
\frac{1}{8}\int_{\R^2}\lambda[ 2 V(x)+(\nabla V,x)]|u_\mu|^2dx
+ \frac{1-\mu}{8}\bigg(2\|u_\mu\|_*^2 +\int_{\R^2}\frac{|x|}{2+|x|}u_\mu^2dx\bigg)\\
&\ \ \ \
+\frac{1}{16 }|u_\mu|_2^4+\frac{\mu}{2 }\int_{\R^2} [ f^{R,\bar{\delta}}(u_\mu)u_\mu-3 F^{R,\bar{\delta}}(u_\mu) ]dx
\end{align*}
implying that $J_{\lambda,\mu}^{R,\bar{\delta}}(u_\mu)>0$, where we have exploited $(h_2^1)$ and $(V_4^1)$.

Finally, the remaining part is to verify $J_{\lambda,\mu}^{R,\bar{\delta}}(u_\mu)\leq c_{\lambda,\mu}^{R,\bar{\delta}}$.
To aim it, we claim that $V_1(u_n)\to V_1(u_\mu)$ as $n\to\infty$. Indeed, in view of Lemma \ref{zz}-(ii)
and $\|u_n\|_{X_\lambda}\leq C$,
\[
 V_1(u_n)=\int_{\R^2}\bigg(\int_{\R^2}\log(2+|x-y|)u^2(y)dy\bigg)u^2(x)dx   \leq 2\|u_n\|_*^2\int_{\R^2}u_n^2dx\leq 2C^2\int_{\R^2}u_n^2dx
\]
jointly with $u_n\to u_\mu$ in $L^2(\R^2)$ yields the claim. By adopting \ref{zz}-(v)
and $(h_{21})$, it follows from the Fatou's lemma that
\begin{align*}
  c_{\lambda,\mu}^{R,\bar{\delta}} &=\liminf_{n\to\infty}\big[ J_{\lambda,\mu}^{R,\bar{\delta}}(u_n)-\frac1\theta
   (J_{\lambda,\mu}^{R,\bar{\delta}})^\prime(u_n)[u_n]\big]\\
   & =\liminf_{n\to\infty}\bigg\{\frac{2\theta}{\theta-2}\|u_n\|_{E_\lambda}^2+\frac{2\theta(1-\mu)}{\theta-2}\|u_n\|_{*}^2+
\frac{\theta-4}{4\theta}[V_1(u_n)-V_2(u_n)]\\
&\ \ \ \ +\frac{\mu}{\theta}
\int_{\R^2}[f^{R,\bar{\delta}}(u_n)u_n-\theta F^{R,\bar{\delta}}(u_n) ]dx\bigg\}\\
& \geq \frac{2\theta}{\theta-2}\|u_\mu\|_{E_\lambda}^2+\frac{2\theta(1-\mu)}{\theta-2}\|u_\mu\|_{*}^2+
\frac{\theta-4}{4\theta}[V_1(u_\mu)-V_2(u_\mu)]\\
&\ \ \ \ +\frac{\mu}{\theta}
\int_{\R^2}[f^{R,\bar{\delta}}(u_\mu)u_\mu-\theta F^{R,\bar{\delta}}(u_\mu) ]dx \\
&= J_{\lambda,\mu}^{R,\bar{\delta}}(u_\mu)-\frac1\theta
   (J_{\lambda,\mu}^{R,\bar{\delta}})^\prime(u_\mu)[u_\mu]=J_{\lambda,\mu}^{R,\bar{\delta}}(u_\mu).
\end{align*}
So, we accomplish the proof of this lemma.
\end{proof}

As a direct consequence of Lemma \ref{propo5}, there exist two sequence
$\{\mu_n\}\subset[\frac12,1]$ and $\{u_n(\mu_n)\}\subset X_\lambda$ still denoted by $\{u_n\}$
such that
\begin{equation}\label{propo5aa}
   \mu_n\to1^-,~J_{\lambda,\mu_n}^{R,\bar{\delta}}(u_n)\triangleq  {\bar{c}}_{\lambda,\mu_n}^{R,\bar{\delta}}
  \in(0, c_{\lambda,\mu_n}^{R,\bar{\delta}}]~\text{and}~
(J_{\lambda,\mu_n}^{R,\bar{\delta}})^\prime(u_n)=0.
\end{equation}
Before presenting the proof of Theorem \ref{maintheorem6},
we need to pull ${\bar{c}}_{\lambda,\mu_n}^{R,\bar{\delta}}$
down to some critical threshold value whence $\bar{\delta}=2$, namely

\begin{lemma}\label{22Propo2}
Under the assumptions of Lemma \ref{propo2}, if in addition we suppose that $(h_4)$,
then there exists a $\bar{\xi}_0=\bar{\xi}_0(R)(\geq \underline{\xi}_0(R))>0$
 such that for all $\xi>\bar{\xi}_0$, there holds
\begin{equation}\label{22Propo2a}
\frac{8(\theta-3)}{\theta-2} {c}_{\lambda,\mu}^{R,2}
+16{K_0^2\kappa^{3}_{\text{GN}}}\big( {c}_{\lambda,\mu}^{R,2}\big)^{\frac32}
<\frac{\pi}{C_{\Xi,b}(1+\gamma+\alpha R^{\tau-2})},~\forall \lambda\geq1,~\mu\in[\frac12,1],
\end{equation}
 where $\gamma>2$, $\theta>3$, $K_0>0$ and $C_{\Xi,b}>0$ are from $(h_3)$, $(h_{21})$,
  Lemmas \ref{zz}-\emph{(iv)} and \ref{imbedding}, respectively.
Moreover, $\kappa_{\text{GN}}>0$ denotes the best constant associated with Gagliardo-Nirenberg inequality.
\end{lemma}

 \begin{proof}
Since ${c}_{\lambda,\mu}^{R,2}$ in non-increasing on $\mu\in[\frac12,1]$ by Proposition \ref{Jeanjean},
it suffices to deduce that ${c}_{\lambda,\frac12}^{R,2}$ satisfies \eqref{22Propo2a}.
On the one hand, we could claim that
${c}_{\lambda,\frac12}^{R,2}\leq \inf\limits_{u\in X_\lambda\backslash\{0\}}\max\limits_{t>0}J_{\lambda,\frac12}^{R,2}(tu)$.
Indeed, this is a direct corollary of Lemma \ref{propo2}-(i) and (ii).
On the other hand, we can follow the calculations in
Lemma \ref{mplevel} to determine some suitable $\underline{\xi}_0=\underline{\xi}_0(R)>0$
$\bar{\xi}_0=\bar{\xi}_0(R)>0$ to ensure Lemma \ref{2Propo2} and this lemma hold true. The proof os finished.
\end{proof}

Now, we are ready to show the proof of Theorem \ref{maintheorem6} as follows.

 \begin{proof}[\textbf{\emph{Proof of Theorem \ref{maintheorem6}}}]
Firstly, we fix the constants $R>0$ and $\lambda\geq \bar{\lambda}_0^\prime(R)\geq1$.
Depending on the above results, we obtain two sequence $\{\mu_n\}\subset [\frac12,1]$
and $\{u_n\}\subset X_\lambda$ satisfying \eqref{propo5aa}.
Obviously, we have that $P_{\lambda,\mu_n}^{R,\bar{\delta}}(u_n)=0$ by Lemma \ref{Pohozaev}
and \eqref{propo5aa}.
So, taking account of $(V_4^1)$ and $(h_{21})$, it holds
\begin{align*}
\bar{c}_{\lambda,\mu_n}^{R,\bar{\delta}} &= J_{\lambda,\mu_n}^{R,\bar{\delta}}(u_n)= J_{\lambda,\mu_n}^{R,\bar{\delta}}(u_n)-\frac14\big[
2(J_{\lambda,\mu_n}^{R,\bar{\delta}})^\prime(u_n)[u_n] -P_{\lambda,\mu_n}^{R,\bar{\delta}}(u_n)\big]    \\
   & =
\frac{1}{8}\int_{\R^2}\lambda[ 2 V(x)+(\nabla V,x)]|u_n|^2dx
+ \frac{1-\mu_n}{8}\bigg(2\|u_n\|_*^2 +\int_{\R^2}\frac{|x|}{2+|x|}u_n^2dx\bigg)\\
&\ \ \ \
+\frac{1}{16 }|u_n|_2^4+\frac{\mu_n}{2 }\int_{\R^2} [ f^{R,\bar{\delta}}(u_n)u_n-3 F^{R,\bar{\delta}}(u_n) ]dx\\
&\geq \frac{1}{16 }|u_n|_2^4+\frac{\theta-3}{4}\int_{\R^2}  F^{R,\bar{\delta}}(u_n) dx
\end{align*}
which indicates that
\begin{equation}\label{maintheorem61}
|u_n|_2^2\leq  4\sqrt{\bar{c}_{\lambda,\mu_n}^{R,\bar{\delta}}}~\text{and}~
\int_{\R^2}  F^{R,\bar{\delta}}(u_n) dx\leq\frac{4\bar{c}_{\lambda,\mu_n}^{R,\bar{\delta}}}{\theta-3}.
\end{equation}
In light of \eqref{propo5aa} and \eqref{maintheorem61},
we apply Lemma \ref{zz}-(iv) and the Gagliardo-Nirenberg inequality to get
\begin{align*}
\bar{c}_{\lambda,\mu_n}^{R,\bar{\delta}}&=\frac{1}{2}\int_{\R^2}[|\nabla u_n|^2+\lambda V(x)|u_n|^2]dx+ \frac12\|u_n\|_*^2 + \frac14V_0(u_n)
- \mu_n\bigg(\int_{\R^2} F^{R,\bar{\delta}}(u_n)dx+ \frac12\|u_n\|_*^2  \bigg) \\
    & \geq \frac{1}{2}\|u_n\|_{E_\lambda}^2+ \frac14\|u_n\|_*^2-\frac14 V_2(u_n)-\frac12\int_{\R^2} F^{R,\bar{\delta}}(u_n)dx\\
&\geq \frac{1}{2}\|u_n\|_{E_\lambda}^2+ \frac14\|u_n\|_*^2-\frac{K_0\kappa^{\frac32}_{\text{GN}}}4|u_n|_2^3|\nabla u_n|_2
-\frac12\int_{\R^2} F^{R,\bar{\delta}}(u_n)dx\\
&\geq \frac{1}{4}\|u_n\|_{E_\lambda}^2+ \frac14\|u_n\|_*^2-4{K_0^2\kappa^{3}_{\text{GN}}}\big(\bar{c}_{\lambda,\mu_n}^{R,\bar{\delta}}\big)^{\frac32}
-\frac{2\bar{c}_{\lambda,\mu_n}^{R,\bar{\delta}}}{\theta-3}.
\end{align*}
In this situation, we conclude that
\begin{equation}\label{22Propo2e}
\|u_n\|_{X_\lambda}^2=
\|u_n\|_{E_\lambda}^2+  \|u_n\|_*^2\leq \frac{8(\theta-3)}{\theta-2}\bar{c}_{\lambda,\mu_n}^{R,\bar{\delta}}
+16{K_0^2\kappa^{3}_{\text{GN}}}\big(\bar{c}_{\lambda,\mu_n}^{R,\bar{\delta}}\big)^{\frac32}.
\end{equation}
Combining Lemma \ref{propo2}-(iii) and Lemma \ref{22Propo2},
$\{\|u_n\|_{X_\lambda}\}$ is uniformly bounded in $n\in \mathbb{N}$. Moreover,
due to Lemma \ref{imbedding}, we deduce that $\{u_n\}$ satisfies all assumptions in
Lemma \ref{compact}. Repeating the proof of Theorem \ref{maintheorem2}, we could verify that, going to a subsequence if necessary,
$u_n\to u$ in $X_\lambda$ as $n\to\infty$. Besides, $u\neq0$ by \eqref{propo5aa} again and $\mu_n\to1^-$ implies that
\[
(J_\lambda^{R,\bar{\delta}}(u))^\prime[\phi]=\lim_{n\to\infty}(J_{\lambda,\mu_n}^{R,\bar{\delta}}(u_n))^\prime[\phi]
=0,~\forall \phi\in C_0^\infty(\R^2)
\]
finishing the proof.
\end{proof}

\subsubsection{\emph{Proof of Theorem \ref{maintheorem7}}}
Let us  continue to use the work space $(X_\lambda,\|\cdot\|_{X_\lambda})$
in this subsection.

Motivated by \cite[Theorem 1.3]{CT1}, we shall take advantage of
the Nehari-Poho\u{z}aev manifold method
to consider Theorem \ref{maintheorem7}, that is, looking for a minimizer of
the minimization problem
\begin{equation}\label{maintheorem7a}
m_{\lambda,V}^{R,\bar{\delta}}\triangleq \inf_{u\in
\mathcal{M}_{\lambda,V}^{R,\bar{\delta}}}J_{\lambda}^{R,\bar{\delta}}(u)
~\text{with}~   \mathcal{M}_{\lambda,V}^{R,\bar{\delta}}=\{u\in X_\lambda \backslash\{0\}:
{G}_\lambda^{R,\bar{\delta}}(u)=0\},
\end{equation}
where ${G}_\lambda^{R,\bar{\delta}}(u)=2(J_\lambda^{R,\bar{\delta}})^\prime(u)[u]-  P_\lambda^{R,\bar{\delta}}(u)$
for each $u\in X_\lambda$ and $P_\lambda^{R,\bar{\delta}}=P_{\lambda,1}^{R,\bar{\delta}}$ in Lemma \ref{Pohozaev}.
Recalling that every critical point $u\in X_\lambda\backslash\{0\}$
satisfies $(J_\lambda^{R,\bar{\delta}})^\prime(u)[u]=0$ and $P_\lambda^{R,\bar{\delta}}(u)=0$,
thereby it is natural to minimize $J_\lambda^{R,\bar{\delta}}$ over $\mathcal{M}_{\lambda,V}^{R,\bar{\delta}}$.

\begin{lemma}\label{unique}
Let $V$ satisfy $(V_1)-(V_3)$ and $(V_{4}^1)-(V_{4}^2)$.
Suppose that the nonlinearity $f$ defined in \eqref{form}
requires $(h_1),(h_3)$ and $(h_{21})-(h_{22})$. Then, for each $u\in X_\lambda\backslash\{0\}$, there exists a unique
$t_u>0$ such that $t_u^2u(t_u \cdot)\in \mathcal{M}_{\lambda,V}^{R,\bar{\delta}}$
for each fixed $R>0$ and $\lambda\geq1$. Moreover
\[
 m_{\lambda,V}^{R,\bar{\delta}}=\inf_{u\in X_\lambda\backslash\{0\}}\max_{t>0} J_{\lambda,V}^{R,\bar{\delta}}(t^2 u(t \cdot)).
\]
\end{lemma}

\begin{proof}
According to the definition of $f^{R,\bar{\delta}}$ defined in \eqref{fR}, we deduce that
$f^{R,\bar{\delta}}(t)/t^2$ is increasing on $t\in\R^+$ and so $[f^{R,\bar{\delta}}(t)t-F^{R,\bar{\delta}}(t)]/t^3$ is increasing on $t\in\R^+$
on $t\in\R^+$ either. Proceeding as the proofs of \cite[Lemmas 4.3, 4.6 and 4.7]{CT1},
we can get the desired result.
\end{proof}

\begin{lemma}\label{manifold}
Let $V$ satisfy $(V_1)-(V_3)$ and $(V_{4}^1)-(V_{4}^2)$.
Suppose that the nonlinearity $f$ defined in \eqref{form}
requires $(h_1),(h_3)$ and $(h_{21})-(h_{22})$, then for each fixed $R>0$ and $\lambda\geq1$,
we conclude that
\begin{itemize}
  \item[\emph{(i)}] There is a constant $\varrho>0$ independent $\lambda\geq1$ such that $\|u\|_{H^1(\R^2)}\geq\varrho$
for all $u\in\mathcal{M}_{\lambda,V}^{R,\bar{\delta}}$;
  \item[\emph{(ii)}] The minimum $m_{\lambda,V}^{R,\bar{\delta}}>0$, where it is defined in \eqref{maintheorem7a}.
\end{itemize}
\end{lemma}

\begin{proof}
(i) We suppose it by a contradiction, namely there exists a sequence $\{u_n\}\subset \mathcal{M}_{\lambda,V}^{R,\bar{\delta}}$
such that $\|u_n\|_{H^1(\R^2)} \to0$. It follows from Lemma \ref{imbedding}, $(V_4^2)$, \eqref{v2} and \eqref{compact5} that
\[
C_1\|u_n\|_{H^1(\R^2)}^2\leq C_2\|u_n\|_{H^1(\R^2)}^4+\frac{C_1}{2}\|u_n\|_{H^1(\R^2)}^2
+C_2(R)\|u_n\|_{H^1(\R^2)}^q
\]
yielding a contradiction if we choose $q>2$.

(ii) Arguing it indirectly, we suppose that $m_{\lambda,V}^{R,\bar{\delta}}=0$.
Let $\{u_n\}\subset X_\lambda$ be a minimizing sequence of $m_{\lambda,V}^{R,\bar{\delta}}$,
that is, $\{u_n\}\subset \mathcal{M}_{\lambda,V}^{R,\bar{\delta}}$ and $J_{\lambda}^{R,\bar{\delta}}(u_n)\to0$
as $n\to\infty$. Due to $(h_2^1)$ and $(V_{4}^1)$, there holds
\begin{align*}
  o_n(1) &=J_\lambda^{R,\bar{\delta}}(u_n)-\frac{ 1}{4}{G}_\lambda^{R,\bar{\delta}}(u_n) \\
    &= \frac{1}{8}\int_{\R^2}\lambda[ 2 V(x)+(\nabla V,x)]|u_n|^2dx
+\frac{1}{16 }|u_n|_2^4+\frac{1}{2 }\int_{\R^2} [ f^{R,\bar{\delta}}(u_n)u_n-3 F^{R,\bar{\delta}}(u_n) ]dx\\
&\geq  \frac{1}{16 }|u_n|_2^4+\frac{\theta-3}{2\theta}\int_{\R^2}  f^{R,\bar{\delta}}(u_n)u_n dx\\
&\geq  \frac{1}{16 }|u_n|_2^4+\frac{\theta-3}{2 }\int_{\R^2}  F^{R,\bar{\delta}}(u_n) dx\geq0
\end{align*}
which reveals that
\[
|u_n|_2=o_n(1),~\int_{\R^2}  f^{R,\bar{\delta}}(u_n)u_n dx=o_n(1)~\text{and}~
\int_{\R^2}F^{R,\bar{\delta}}(u_n) dx=o_n(1).
\]
Using ${G}_\lambda^{R,\bar{\delta}}(u_n)=0$ again,
via the above formula, we obtain
\begin{align*}
 2\int_{\R^2} |\nabla u_n|^2dx &\leq V_2(u_n)+\frac14|u_n|_2^4+2 \int_{\R^2} [f^{R,\bar{\delta}}(u_n)u_n- F^{R,\bar{\delta}}(u_n) ]dx \\
    & =V_2(u_n)+o_n(1) \leq  K_0\kappa^{\frac32}_{\text{GN}} |u_n|_2^3|\nabla u_n|_2+o_n(1)\\
&\leq \frac12|\nabla u_n|_2^2+\frac{1}{2}K_0^2\kappa^{3}_{\text{GN}}|u_n|_2^6+o_n(1)
=\frac12|\nabla u_n|_2^2 +o_n(1)
\end{align*}
 which is impossible by Point-(i). The proof of this lemma is finished.
\end{proof}

\begin{lemma}\label{5bounded}
Let $V$ satisfy $(V_1)-(V_3)$ and $(V_{4}^1)-(V_{4}^2)$.
Suppose that the nonlinearity $f$ defined in \eqref{form}
requires $(h_1),(h_3)$ and $(h_{21})-(h_{22})$, then for each fixed $R>0$ and $\lambda\geq1$,
any minimizing sequence $\{u_n\}\subset X_\lambda$ satisfies
\begin{equation}\label{5bounded1}
\|u_n\|_{E_\lambda}^2 \leq \frac{8(\theta-3)}{\theta-2}m_{\lambda,V}^{R,\bar{\delta}}
+16{K_0^2\kappa^{3}_{\text{GN}}}\big(m_{\lambda,V}^{R,\bar{\delta}}\big)^{\frac32}+o_n(1).
\end{equation}
\end{lemma}

\begin{proof}
Let $\{u_n\}\subset X_\lambda$ be a minimizing sequence of $m_{\lambda,V}^{R,\bar{\delta}}$,
that is, $\{u_n\}\subset \mathcal{M}_{\lambda,V}^{R,\bar{\delta}}$ and $J_{\lambda}^{R,\bar{\delta}}(u_n)\to m_{\lambda,V}^{R,\bar{\delta}}$
as $n\to\infty$. Recalling
${G}_\lambda^{R,\bar{\delta}}(u_n)=2(J_\lambda^{R,\bar{\delta}})^\prime(u_n)[u_n]-  P_\lambda^{R,\bar{\delta}}(u_n)$,
it is very similar to
\eqref{maintheorem61} that
\[
|u_n|_2^2\leq 4\sqrt{m_{\lambda,V}^{R,\bar{\delta}}+o_n(1)}~\text{and}~
\int_{\R^2}  F^{R,\bar{\delta}}(u_n) dx\leq \frac{2m_{\lambda,V}^{R,\bar{\delta}}}{\theta-3 }+o_n(1).
\]
According to this formula, using Lemma \ref{zz}-(iv) and the Gagliardo-Nirenberg inequality
\begin{align*}
 m_{\lambda,V}^{R,\bar{\delta}}+o_n(1) &\geq  \frac{1}{2}\|u_n\|_{E_\lambda}^2-\frac14V_2(u_n)- \int_{\R^2} F^{R,\bar{\delta}}(u_n)dx \\
   & \geq \frac{1}{2}\|u_n\|_{E_\lambda}^2-\frac{K_0\kappa^{\frac32}_{\text{GN}}}4|u_n|_2^3|\nabla u_n|_2- \frac{2}{\theta-3 }
[m_{\lambda,V}^{R,\bar{\delta}}+o_n(1)] \\
   &\geq \frac{1}{4}\|u_n\|_{E_\lambda}^2- 4{K_0^2\kappa^{3}_{\text{GN}}}\big[{m_{\lambda,V}^{R,\bar{\delta}}+o_n(1)}\big]^{\frac32}
     - \frac{2}{\theta-3 }[m_{\lambda,V}^{R,\bar{\delta}}+o_n(1)]
\end{align*}
showing the desired result. The proof this lemma is finished.
\end{proof}

Arguing as before, we have to verify that the minimum ${m}_{\lambda,V}^{R,\bar{\delta}}$
can be controlled suitably.

\begin{lemma}\label{control}
Let $V$ satisfy $(V_1)-(V_3)$. Suppose that $f$ defined by \eqref{form}
satisfies $(h_1)$ and $(h_3)$, then for all fixed $R>0$ and $\lambda\geq1$,
there is a constant $\bar{M}^{R,\delta}>0$ such that ${m}_{\lambda,V}^{R,{\delta}}\leq \bar{M}^{R,\delta}$.
Moreover, if in addition we suppose that $(h_4)$,
then there exists a $\tilde{\xi}_0=\tilde{\xi}_0(R)>0$
 such that for all $\xi>\tilde{\xi}_0$, there holds
\[
\frac{8(\theta-3)}{\theta-2} {m}_{\lambda,V}^{R,2}
+16{K_0^2\kappa^{3}_{\text{GN}}}\big( {m}_{\lambda,V}^{R,2}\big)^{\frac32}
<\frac{\pi}{C_{\Xi,b}(1+\gamma+\alpha R^{\tau-2})},~\forall \lambda\geq1.
\]
\end{lemma}

\begin{proof}
The proof is very similar to that of Lemma \ref{22Propo2}, so we omit it here.
\end{proof}

\begin{proof}[\textbf{\emph{Proof of Theorem \ref{maintheorem7}}}]
We divide the proof into intermediate steps.

\medspace
{\sc Step 1:} The minimization problem ${m}_{\lambda,V}^{R,\bar{\delta}}$ in \eqref{maintheorem7a} can be attained.

Let $\{u_n\}\subset X_\lambda$ be a minimizing sequence of $m_{\lambda,V}^{R,\bar{\delta}}$,
that is, $\{u_n\}\subset \mathcal{M}_{\lambda,V}^{R,\bar{\delta}}$ and $J_{\lambda}^{R,\bar{\delta}}(u_n)\to0$
as $n\to\infty$. In view of Lemma \ref{5bounded}, $\{\|u_n\|_{E_\lambda}\}$
is uniformly bounded in $n\in \mathbb{N}$. Using Lemma \ref{5bounded} again,
thanks to Lemma \ref{control}, we can follow step by step in Lemmas
\ref{Vanishing}, \ref{Non-Vanishing} and \ref{bounded}
to find $\tilde{\lambda}(R)$ if $\bar{\delta}=\delta$,
or $\tilde{\lambda}^\prime(R)$ if ${\delta}=2$,
to assure that $\{\|u_n\|_{X_\lambda}\}$ is uniformly bounded in $n\in \mathbb{N}$
for all $\lambda>\tilde{\lambda}(R)$ if $\bar{\delta}=\delta$, or
$\lambda>\tilde{\lambda}^\prime(R)$ if ${\delta}=2$, respectively.
Repeating the proof of Theorem \ref{maintheorem2}, up to a subsequence,
 there is a $u\in X_\lambda$ such that $u_n\to u$ in $X_\lambda$.
So, the nontrivial function $u$ is the desired attained function.

\medspace
{\sc Step 2:}  The attained function of ${m}_{\lambda,V}^{R,\bar{\delta}}$ is indeed a
nontrivial nonnegative critical point of $J_{\lambda,V}^{R,\bar{\delta}}$.

We refer the interested reader to \cite[Lemma 4.11]{CT1} for the detailed proof since
 there is no essential difference.
\end{proof}

\subsection{Proofs of Theorems \ref{maintheorem4}
and \ref{maintheorem5}}
Let us note that, combining Lemma \ref{propo2}-(iii), Lemma \ref{22Propo2} and \eqref{22Propo2e},
we can still derive \eqref{Ia1} and \eqref{IIa4} for the proof of Theorem \ref{maintheorem4}. On the other hand, with the help of
  \eqref{5bounded1} and Lemma \ref{control},
we conclude \eqref{Ia1} and \eqref{IIa4} for the proof of Theorem \ref{maintheorem5}.
As a consequence, by repeating the calculations in Section \ref{Linftyestimate},
 we would accomplish the proofs of Theorems \ref{maintheorem4}
and \ref{maintheorem5} immediately.

\vskip 5mm
\noindent\textbf{Conflict of interest.} The authors have no competing interests to declare related to this paper.

\bigskip

\end{document}